\documentclass[12pt]{amsart}
\usepackage{tikz-cd}
\usepackage{enumitem}
\setlist[description]{leftmargin=\parindent,labelindent=\parindent}
\usepackage{amssymb}
\usepackage{amsmath,amscd}
\usepackage{mathrsfs}
\usepackage{url}
\usepackage{mathtools}
\usepackage{cite}
\usepackage{fullpage}
\usepackage[hidelinks]{hyperref}
\usepackage{verbatim}
\usepackage{comment}
\usepackage{fancyvrb}
\usepackage{fvextra}
\usepackage{caption}
\usepackage{tikz}
\usepackage{tkz-graph}
\usetikzlibrary{matrix}
 \usetikzlibrary{shapes}
\usetikzlibrary{arrows,decorations.markings}
\usepackage{tikz-cd}
\usepackage{stmaryrd}

\pgfarrowsdeclare{bad to}{bad to}
{
  \pgfarrowsleftextend{-2\pgflinewidth}
  \pgfarrowsrightextend{\pgflinewidth}
}
{
  \pgfsetlinewidth{0.8\pgflinewidth}
  \pgfsetdash{}{0pt}
  \pgfsetroundcap
  \pgfsetroundjoin
  \pgfpathmoveto{\pgfpoint{-3\pgflinewidth}{4\pgflinewidth}}
  \pgfpathcurveto
  {\pgfpoint{-2.75\pgflinewidth}{2.5\pgflinewidth}}
  {\pgfpoint{0pt}{0.25\pgflinewidth}}
  {\pgfpoint{0.75\pgflinewidth}{0pt}}
  \pgfpathcurveto
  {\pgfpoint{0pt}{-0.25\pgflinewidth}}
  {\pgfpoint{-2.75\pgflinewidth}{-2.5\pgflinewidth}}
  {\pgfpoint{-3\pgflinewidth}{-4\pgflinewidth}}
  \pgfusepathqstroke
}

\newtheorem{thm}{Theorem}[section]
\newtheorem{prop}[thm]{Proposition}
\newtheorem{lem}[thm]{Lemma}
\newtheorem{cor}[thm]{Corollary}
\newtheorem{conj}[thm]{Conjecture}

\theoremstyle{definition}
\newtheorem{definition}[thm]{Definition}

\newtheorem{rem}[thm]{Remark}

\numberwithin{equation}{section}
\newcommand{\semis}{\mathrm{ss}}
\newcommand{\s}{\mathsf{S}}
\renewcommand{\l}{\mathsf{L}}

\renewcommand{\L}{\mathcal{L}}
\newcommand{\U}{\mathcal{U}}

\newcommand{\Y}{\mathcal{Y}}
\newcommand{\Q}{\mathcal{Q}}

\newcommand{\V}{\mathcal{V}}

\newcommand{\B}{\mathcal{B}}

\newcommand{\zz}{\mathbb{Z}}

\newcommand{\qq}{\mathbb{Q}}

\newcommand{\C}{\mathcal{C}}
\newcommand{\W}{\mathcal{W}}
\renewcommand{\ss}{\mathbb{S}}

\newcommand{\M}{\mathcal{M}}
\newcommand{\p}{\mathbb{P}}
\newcommand{\pp}{\mathbb{P}}

\renewcommand{\H}{\mathcal{H}}
\newcommand{\Xb}{\overline{X}}
\newcommand{\F}{\mathcal{F}}
\renewcommand{\P}{\mathcal{P}}
\DeclareMathOperator{\cl}{cl}
\newcommand{\E}{\mathcal{E}}
\renewcommand{\O}{\mathcal{O}}

\DeclareMathOperator{\ct}{ct}

\renewcommand{\tilde}{\widetilde}

\DeclareMathOperator{\Aut}{Aut}
\DeclareMathOperator{\GL}{GL}
\DeclareMathOperator{\Gr}{Gr}
\DeclareMathOperator{\BSO}{BSO}
\renewcommand{\SO}{\mathrm{SO}}
\DeclareMathOperator{\OG}{OG}
\DeclareMathOperator{\SL}{SL}

\DeclareMathOperator{\Sym}{Sym}

\DeclareMathOperator{\BSL}{BSL}
\DeclareMathOperator{\BPGL}{BPGL}
\DeclareMathOperator{\BGL}{BGL}

\DeclareMathOperator{\Hom}{Hom}

\renewcommand{\gg}{\mathbb{G}}

\newcommand{\lstw}{\mathsf{L}\mathsf{S}_{12}}

\usepackage{wrapfig}

\newcommand{\Mb}{\overline{\M}}

\renewcommand{\S}{\mathcal{S}}

\title{Extensions of tautological rings and motivic structures in the cohomology of $\Mb_{g,n}$}
\author{Samir Canning}
\author{Hannah Larson}
\author{Sam Payne}
\thanks{
S.C. was supported by a Hermann-Weyl-Instructorship from the Forschungsinstitut für Mathematik at ETH Z\"urich. This research was partially conducted during the period H.L. served as a Clay Research Fellow and as a guest of the Forschungsinstitut für Mathematik at ETH Z\"urich. S.P. was supported in part by NSF grants DMS--2302475 and DMS--2053261. }

\begin{document}

\begin{abstract}
    We study collections of subrings of $H^*(\Mb_{g,n})$ that are closed under the tautological operations that map cohomology classes on moduli spaces of smaller dimension to those on moduli spaces of larger dimension and contain the tautological subrings. Such extensions of tautological rings are well-suited for inductive arguments and flexible enough for a wide range of applications.  In particular, we confirm predictions of Chenevier and Lannes for the $\ell$-adic Galois representations and Hodge structures that appear in $H^k(\Mb_{g,n})$ for $k = 13$, $14$, and $15$.  We also show that $H^4(\Mb_{g,n})$ is generated by tautological classes for all $g$ and $n$, confirming a prediction of Arbarello and Cornalba from the 1990s.  In order to establish the final base cases needed for the inductive proofs of our main results, we use Mukai's construction of canonically embedded pentagonal curves of genus 7 as linear sections of an orthogonal Grassmannian and a decomposition of the diagonal to show that the pure weight cohomology of $\M_{7,n}$ is generated by algebraic cycle classes, for $n \leq 3$.
\end{abstract}

\maketitle

\section{Introduction}

The moduli spaces of stable curves $\Mb_{g,n}$ are smooth and proper over the integers, and this implies strong restrictions on the motivic structures, such as $\ell$-adic Galois representations,  that can appear in $H^*(\Mb_{g,n})$. Widely believed conjectures regarding analytic continuations and functional equations for $L$-functions lead to precise predictions, by Chenevier and Lannes, about which such structures can appear in degrees less than or equal to 22 \cite[Theorem~F]{ChenevierLannes}. These predictions are consistent with all previously known results on $H^*(\Mb_{g,n})$.  Recent work inspired by these predictions confirms their correctness in all degrees less than or equal to 12 \cite{BergstromFaberPayne, CanningLarsonPayne}. Here we introduce new methods to systematically study the motivic structures in $H^k(\Mb_{g,n})$ for $k > 12$ and confirm these predictions in degrees $13$, $14$, and $15$.

Throughout, we write $H^*(X)$ for the rational singular cohomology of a scheme or Deligne-Mumford stack $X$ endowed with its associated Hodge structure or $\ell$-adic Galois representation and $H^*(X)^{\semis}$ for its semi-simplification. Let $\l := H^2(\p^1)$ and $\s_{12} := H^{11}(\Mb_{1,11})$.

\begin{thm} \label{thm:131415}
    For all $g$ and $n$, we have $H^{13}(\Mb_{g,n})^{\semis} \cong \bigoplus \lstw$ and $H^{14}(\Mb_{g,n})^{\semis} \cong \bigoplus \l^{7}$. Moreover, for $g \geq 2$, we have
$H^{15} (\Mb_{g,n})^{\semis} \cong \bigoplus \l^2 \s_{12}$.

\end{thm}

\noindent Theorem~\ref{thm:131415} confirms the predictions of Chenevier and Lannes for motivic weights $k \leq 15$.  Note that the Hodge structure on the cohomology of a smooth and proper Deligne--Mumford stack such as $\Mb_{g,n}$ is semi-simple, so the semi-simplification in Theorem~\ref{thm:131415} is relevant only when considering $\ell$-adic Galois representations.

The proof of Theorem~\ref{thm:131415} uses the inductive structure of the boundary of the moduli space and the maps induced by tautological morphisms between moduli spaces, as do the proofs of the precursor results mentioned above.  Recall that the collection of tautological rings $RH^*(\Mb_{g,n}) \subset H^*(\Mb_{g,n})$ is the smallest collection of subrings that is closed under pushforward and pullback for the tautological morphisms induced by gluing, forgetting, and permuting marked points. For many inductive arguments, including those used here, it suffices to consider the operations that produce cohomology classes on moduli spaces of larger dimension from those on moduli spaces of smaller dimension. 

\begin{definition}
A \emph{semi-tautological extension} (STE) is a collection of subrings $S^*(\Mb_{g,n})$ of $H^*(\Mb_{g,n})$ that contains the tautological subrings $RH^*(\Mb_{g,n})$ and is closed under pullback by forgetting and permuting marked points and under pushforward for gluing marked points.
\end{definition} 

\noindent Examples of STEs include the trivial extension $RH^*$, the full cohomology rings $H^*$, and the collection of subrings generated by algebraic cycle classes.  Not every STE is closed under the additional tautological operations induced by push-forward for forgetting marked points and pullback for gluing marked points. However, the main examples we study here are indeed closed under all of the tautological operations. See Proposition~\ref{ts}.  

Note that any intersection of STEs is an STE. An STE is \emph{finitely generated} if it is the smallest STE that contains a given finite subset (or, equivalently, the union of finitely many $\qq$-vector subspaces) of $\coprod_{g,n} H^*(\Mb_{g,n})$. A finitely generated STE is suitable for combinatorial study via algebraic operations on decorated graphs. See \cite{PandharipandePixtonZvonkine, Pixton} for discussions of the graphical algebra underlying the tautological ring, and \cite{PayneWillwacher21, PayneWillwacher23} for applications of such operadic methods, with not necessarily tautological decorations, to the weight spectral sequence for $\M_{g,n}$.  Every STE that we consider is \emph{motivic}, meaning that $S^*(\Mb_{g,n}) \subset H^*(\Mb_{g,n})$ is a sub-Hodge structure and its base change to $\mathbb{Q}_\ell$ is preserved by the Galois action.

\begin{thm}\label{thm:fgSTE} For any fixed degree $k$, the STE generated by 
\[\left\{H^{k'}(\Mb_{g',n'}) : k' \leq k, g' < \tfrac{3}{2}k' + 1, \  n' \leq k', \ 4g' - 4 + n' \geq k' \right\}\]
contains $H^k(\Mb_{g,n})$ for all $g$ and $n$. 
\end{thm}

In particular, there is a finitely generated STE that contains $H^k(\Mb_{g,n})$ for all $g$ and $n$.

\begin{cor}
For each $k$, there are only finitely many isomorphism classes of simple Hodge structures (respectively, $\ell$-adic Galois representations) in $\bigoplus_{g,n} H^k(\Mb_{g,n})^{\semis}$.
\end{cor}

We developed the notion of STEs to study nontrivial extensions of tautological rings, such as the STE generated by $H^{11}(\Mb_{1,11})$, but the same methods also yield new results on the tautological ring itself.  We apply the explicit bounds in Theorem~\ref{thm:fgSTE}, together with new tools and results for small $g$, $n$, and $k$ to prove the following.  At the level of $\qq$-vector spaces, we identify $H_k(\Mb_{g,n})$ with $H^{2d_{g,n} - k}(\Mb_{g,n})$, where $d_{g,n} := 3g-3+n$ is the dimension of $\Mb_{g,n}$.  Similarly, when $S^*$ is an STE we write $S_k(\Mb_{g,n})$ for the $\qq$-vector space $S^{2d_{g,n} - k}(\Mb_{g,n})$.

\begin{thm} \label{thm:taut}
        The tautological ring $RH^*(\Mb_{g,n})$ contains \begin{enumerate}
        \item \label{it:k4}$H^4(\Mb_{g,n}),$ for all $g$ and $n$,
        \item  $H^6(\Mb_{g,n}),$ for $g \geq 10$,
        \item \label{it:homology} $H_{k}(\Mb_{g,n}),$ for even $k \leq 14$, for all $g$ and $n$.
    \end{enumerate}
\end{thm}

\begin{thm} \label{thm:13}
The STE generated by $H^{11}(\Mb_{1,11})$ contains $H_{13}(\Mb_{g,n})$ for all $g$ and $n$. 
\end{thm}

\begin{thm} \label{thm:15}
The STE generated by $H^{11}(\Mb_{1,11})$ and $H^{15}(\Mb_{1,15})$ contains $H_{15}(\Mb_{g,n})$ for all $g$ and $n$. 
\end{thm}

Theorem~\ref{thm:taut}\eqref{it:k4} confirms a prediction of Arbarello and Cornalba from the 1990s; they proposed that their inductive method used to prove that $H^2(\Mb_{g,n})$ is tautological should also apply in degree 4 \cite[p.~1]{ArbarelloCornalba}. Shortly thereafter, Polito confirmed that $H^4(\Mb_{g,n})$ is tautological for $g \geq 8$ \cite{Polito}, but the general case remained open until now.

Theorem~\ref{thm:taut}\eqref{it:homology} implies that $H^k(\Mb_{g,n}) \cong \bigoplus \l^{k/2}$ for even $k \leq 14$.  The work of Chenevier and Lannes predicts that $H^k(\Mb_{g,n})$ should also be isomorphic to $\bigoplus \l^{k/2}$ for $k = 16, 18, 20$.  The Hodge and Tate conjectures then predict that these groups are generated by algebraic cycle classes. However,  generation by algebraic cycle classes is an open problem except in the cases where $H^k(\Mb_{g,n})$ is  known to be generated by tautological cycle classes. 

By Theorems~\ref{thm:13} and \ref{thm:15}, $H_{13}(\Mb_{g,n})$ and $H_{15}(\Mb_{g,n})$ lie in STEs generated by cohomology from genus $1$ moduli spaces. On the other hand, $H_{17}(\Mb_{g,n})$ requires genus $2$ data because $W_{17}H^{17}(\M_{2,14})\neq 0$ \cite{FPhandbook}.

 We note that any STE is top-heavy, in the sense that $\dim S^k(\Mb_{g,n}) \leq \dim S_k(\Mb_{g,n})$ for $k \leq \dim \Mb_{g,n}$ (cf. \cite{PetersenTommasi}).  Moreover, the dimensions in even and odd degrees are unimodal.  This is because any STE contains $RH^*$ and hence contains an ample class. Multiplication by a suitable power of the ample class gives an injection from $S^k(\Mb_{g,n})$ to $S_k(\Mb_{g,n})$. 

One could show that $H^k(\Mb_{g,n})$ is generated by tautological cycles in the remaining cases covered by Theorem~\ref{thm:taut}\eqref{it:homology} by showing that the pairing on $RH^k(\Mb_{g,n}) \times RH_k(\Mb_{g,n})$ is perfect, e.g., using \textsf{admcycles}  \cite{admcycles}. However, computational complexity prevents meaningful progress by brute force. Note that Petersen and Tommasi showed that this pairing is not perfect in general for $k \geq 22$ \cite{PetersenTommasi,Petersen}. Graber and Pandharipande had previously shown that $H^{22}(\Mb_{2,20})$ contains an algebraic cycle class that is not tautological \cite{GraberPandharipande}.  

\begin{conj} \label{lowk}
    The tautological ring $RH^*(\Mb_{g,n})$ contains $H^k(\Mb_{g,n})$ 
    for even $k \leq 20$.  
\end{conj}

The results above show that Conjecture~\ref{lowk} is true for $k \leq 4$, and for $k = 6$ and $g \geq 10$. For $k \leq 14$, the conjecture is true if and only if the pairing $RH^k(\Mb_{g,n}) \times RH_k(\Mb_{g,n})$ is perfect.  The examples of Graber--Pandharipande and Petersen--Tommasi show that the conjectured bound of $k \leq 20$ is the best possible. 
As further evidence for Conjecture \ref{lowk}, we note that
the Arbarello--Cornalba induction together with known base cases implies the vanishing of $H^{16,0}(\Mb_{g,n})$ and $H^{18,0}(\Mb_{g,n})$ for all $g$ and $n$, as recently observed by Fontanari  \cite{Fontanari}.

\medskip

The inductive arguments used to study $H^{k}(\Mb_{g,n})$ for all $g$ and $n$ rely on understanding base cases $H^{k'}(\Mb_{g',n'})$ where $g'$ and $n'$ are small relative to $k$. As $k$ grows, more base cases are needed. With the exception of $H^{11}(\Mb_{1,11})$, all base cases required for previous work with $k \leq 12$ have been pure Hodge--Tate \cite{ArbarelloCornalba,BergstromFaberPayne,CanningLarson789}. Substantial work went into establishing these base cases via point counts and other methods. When $k \geq 13$, the problem becomes fundamentally more difficult, as an increasing number of the required base cases are \emph{not} pure Hodge--Tate. In particular, previous techniques for handling base cases do not apply. 

\medskip

The advances presented here depend on a new technique for controlling the pure weight cohomology of $\M_{g,n}$. A space $X$ has the \emph{Chow--K\"unneth generation Property (CKgP)} if the tensor product map on Chow groups $A_*(X) \otimes A_*(Y) \to A_*(X \times Y)$ is surjective for all $Y$. If $X$ is smooth, proper, and has the CKgP, then the cycle class map is an isomorphism. However, in several of the base cases needed for our arguments, the smooth and proper moduli space $\Mb_{g,n}$ has odd cohomology and hence does not have the CKgP. Nevertheless, we show that the open moduli spaces $\M_{g,n}$ do have the CKgP for the relevant pairs $(g,n)$. To apply this in the proof of our main results, the key new  technical statement is Lemma~\ref{cycleclass}, which says that if $X$ is smooth and has the CKgP, then $W_k H^k(X)$ is algebraic. This extension of the aforementioned result on the cycle class map to spaces that are not necessarily proper is essential for controlling the motivic structures that appear in $H^k(\Mb_{g,n})$ for $k \geq 13$. For example, using that $\M_{3,n}$ has the CKgP for $n \leq 11$ \cite[Theorem 1.4]{CL-CKgP}, we determine the Hodge structures and Galois representations that appear in $H^*(\Mb_{3,n})$, for $n \leq 11$.

\begin{thm} \label{thm:g3}
For $n \leq 11$, $H^*(\Mb_{3,n})^\semis$ is a polynomial in $\l$ and $\s_{12}$.
\end{thm}

Bergstr\"om and Faber recently used point counting techniques to compute the cohomology of $\Mb_{3,n}$ as an $\ss_n$-equivariant Galois representation for $n \leq 14$.  For $n \geq 9$, these computations are conditional on the assumption that the only $\ell$-adic Galois representations appearing are those from the list of Chenevier and Lannes \cite{BergstromFaber}. Theorem~\ref{thm:g3} unconditionally confirms the calculations of Bergstr\"om and Faber for $n = 9$, $10$, and $11$.

In order to prove Theorems~\ref{thm:131415}, \ref{thm:taut}\eqref{it:homology}, and \ref{thm:15}, our inductive arguments require several base cases beyond  what was already in the literature. In particular, we prove the following results in genus $7$, which are also of independent interest.

\begin{thm} \label{thm:genus7}
For $n \leq 3$, the moduli space $\M_{7,n}$ has the CKgP and $R^*(\M_{7,n}) = A^*(\M_{7,n})$.
\end{thm}

\noindent Here, $R^*$ denotes the subring of the Chow ring $A^*$ generated by tautological cycle classes.  Previous results proving that $\M_{g,n}$ has the CKgP and $R^*(\M_{g,n}) = A^*(\M_{g,n})$ for small $g$ and $n$ have primarily relied on corresponding results for Hurwitz spaces with marked points \cite{CL-CKgP}. Unfortunately, the numerics for degree $5$ covers prevented this technique from working with marked points. Here, we take a new approach to the pentagonal locus, using a modification of Mukai's construction \cite{Mukai7} that includes markings.

\medskip

\noindent \textbf{Acknowledgements.} We thank Jonas Bergstr\"om, Carel Faber, Rahul Pandharipande, Dan Petersen, Kartik Prasanna, Burt Totaro, and Thomas Willwacher for helpful conversations related to this work.  We especially thank Dan Petersen for his comments that led to an improved version of Lemma \ref{blue}.

\section{Preliminaries}

In this section, we establish notation and terminology that we will use throughout the paper and discuss a few basic examples of STEs. We also  recall previously known facts about the cohomology groups of moduli spaces, especially in genus 0, 1, and 2, that will be used in the base cases of our inductive arguments.

\subsection{Preliminaries on STEs}
Recall that an STE is, by definition, closed under the tautological operations given by pushing forward from the boundary or pulling back from moduli spaces with fewer marked points. Let $\widetilde{\partial \M}_{g,n}$ denote the normalization of the boundary. 
The sequence
\[ H^{k-2}(\widetilde{\partial \M}_{g,n}) \to H^{k}(\Mb_{g,n}) \to W_k H^k(\M_{g,n})\rightarrow 0\]
is right exact. 

Let $\pi_i \colon \M_{g,n} \to \M_{g,n-1}$ be the tautological morphism forgetting the $i$th marking and let
\[
\Phi^k_{g,n} := \pi_1^*W_kH^k(\M_{g,n-1}) + \cdots + \pi_n^*W_kH^k(\M_{g,n-1}) \subset W_kH^k(\M_{g,n}).
\]
The following lemma is a cohomological analogue of the ``filling criteria" in \cite[Section 4]{CL-CKgP}. 

We consider the partial order in which $(g', n') \prec (g, n)$ if $g' \leq g$, $2g' + n' \leq 2g + n$, and $(g',n') \neq (g, n)$. The moduli space $\Mb_{g,n}$ is stratified according to the topological types of stable curves, and each stratum of the boundary is a finite quotient of a product of moduli spaces $\M_{g',n'}$ such that $(g',n') \prec (g,n)$. Recall that we write $d_{g,n} := 3g - 3 + n$.

\begin{lem}
\label{fc}
Let $S^*$ be an STE and let $2g - 2 + n > 0$. If the canonical map  
\begin{equation} \label{check} S^{k'}(\Mb_{g',n'}) \to W_{k'}H^{k'}(\M_{g',n'})/ \big(\Phi^{k'}_{g',n'} + RH^{k'}(\M_{g',n'}) \big)
\end{equation}
is surjective for $(g',n',k') = (g, n, k)$ and all $(g', n', k')$ satisfying
\begin{equation} \label{ih} (g', n') \prec (g, n), \qquad 2d_{g',n'} - k' \leq 2d_{g,n} - k \qquad \text{and} \qquad k' \leq k-2
\end{equation}
then $S^k(\Mb_{g,n}) = H^k(\Mb_{g,n})$.
\end{lem}

\begin{proof}
The proof is by induction on $g$ and $n$. Consider the diagram
\begin{equation} \label{Sd}
\begin{tikzcd}
S^{k-2}(\tilde{\partial \M_{g,n}}) \arrow{d}{\alpha} \arrow{r} & S^{k}(\Mb_{g,n}) \arrow{d}{\beta} \arrow{dr}{\phi} \\
H^{k-2}(\tilde{\partial \M_{g,n}})\arrow{r} & H^k(\Mb_{g,n})\arrow{r} & W_k H^k(\M_{g,n}) \arrow{r} & 0.
\end{tikzcd}
\end{equation}

Here, we extend $S^*$ to $\tilde{\partial \M_{g,n}}$ in the natural way, by summing over components, using the K\"unneth formula,  and taking invariants under automorphisms of the dual graph.  We claim that $\alpha$ is an isomorphism. By the K\"unneth formula, the domain of $\alpha$ is a sum of tensor products of $S^{\ell}(\Mb_{\gamma,\nu})$ with $\ell \leq k - 2$ and $(\gamma, \nu) \prec (g,n)$. Furthermore, by considering dimensions of the cycles involved, we must also have $2d_{\gamma,\nu} - \ell \leq 2d_{g,n} - k$. This is because  we must have $k -2- \ell \leq 2(d_{g,n} - 1 - d_{\gamma,\nu})$, so that the degree of the other K\"unneth component does not exceed its real dimension. Now, suppose we are given $(g',n',k')$ that satisfy
\[(g', n') \prec (\gamma,\nu), \qquad 2d_{g',n'} - k' \leq 2d_{\gamma,\nu} - \ell \qquad \text{and} \qquad k' \leq \ell-2.\]
Then $(g',n',k')$ satisfies \eqref{ih}, so  $S^{\ell}(\Mb_{\gamma, \nu}) = H^\ell(\Mb_{\gamma,\nu})$ by induction.  This proves the claim.

Next, by induction we have $H^k(\Mb_{g,n-1}) = S^k(\Mb_{g,n-1})$, so the image of 
\[\pi_i^*\colon H^k(\Mb_{g,n-1}) \to H^k(\Mb_{g,n})\] is contained in $S^k(\Mb_{g,n})$ for all $i$. Hence, $\Phi^k_{g,n}$ is contained in the image of $\phi$. The tautological classes $RH^k(\M_{g,n}) \subseteq W_k H^k(\M_{g,n})$ are also contained in the image of $\phi$ by definition.
Thus, the surjectivity of 
\[S^k(\Mb_{g,n}) \to W_kH^k(\M_{g,n})/\left(\Phi^k_{g,n} + RH^k(\M_{g,n})\right)\]
implies that $\phi$ is surjective. Hence, $\beta$ is also surjective, as desired.
\end{proof}

In order to apply Lemma \ref{fc}, we need results that help us understand generators for $W_kH^k(\M_{g,n})/\left(\Phi^k_{g,n} + RH^k(\M_{g,n})\right)$. This is the topic of Sections \ref{toprestrictions} and \ref{agp}.

\subsection{Pure cohomology in genus 1 and 2} \label{pg12}
On $\M_{0,n}$, the only nonvanishing pure cohomology is $W_0 H^0(\M_{0,n})$. Below, we present the complete classification of pure cohomology in genus $1$, which is due to Getzler \cite{Getzler}. In genus $2$, we present a classification in low cohomological degree, following Petersen \cite{Petersenlocalsystems}.

\subsubsection{Genus 1} The following statement appeared in \cite[Proposition 7]{Petersenappendix}; the proof there is omitted. Here, we include the proof (and corrected statement), which was explained to us by Dan Petersen.  Let $\mathsf{S}_{k+1}$ denote the weight $k$ structure associated to the space of cusp forms of weight $k+1$ for $\SL_2(\zz)$. Let $\pi \colon E \to \M_{1,1}$ be the universal curve and let $\mathbb{V}$ be the local system $R^1\pi_*\qq$. By the Eichler--Shimura correspondence, we have $\mathsf{S}_{k+1} = W_{k}H^1(\M_{1,1}, \mathbb{V}^{\otimes {k-1}})$. We will see that the latter is identified with $W_k H^k(\M_{1,k})$. More generally, we have the following. Given a partition $\lambda$ of $n$, let $V_\lambda$ be the associated Specht module representation of $\mathbb{S}_n$.

\begin{prop}\label{genus1weightk}
For $n\geq k,$ a basis for
$W_kH^k(\M_{1,n})$ is given by the $\binom{n-1}{k-1}$ pullbacks from $W_kH^k(\M_{1,A})$ where $A$ runs over all subsets of $\{1,\ldots,n\}$ of size $k$ such that $1 \in A$. Consequently,
there is an $\ss_n$-equivariant isomorphism 
\[
W_k H^k(\M_{1,n})\cong \mathsf{S}_{k+1}\otimes V_{n-k+1,1^{k-1}}.
\]
For $n<k$, we have
$W_k H^k(\M_{1,n})=0$.
\end{prop}

\begin{proof}
Let $\pi \colon E\rightarrow \M_{1,1}$ denote the universal elliptic curve, and $\sigma \colon \M_{1,1} \to E$ the section. Associated to the open embedding $\M_{1,n}\hookrightarrow E^{n-1}$, we have a right exact sequence 
\begin{equation}\label{wkres}
\bigoplus W_{k-2}H^{k-2}(E^{n-2})\rightarrow W_k H^k(E^{n-1})\rightarrow W_k H^k(\M_{1,n})\rightarrow 0.
\end{equation}
Since $f\colon E^{n-1}\rightarrow \M_{1,1}$ is smooth and proper, the Leray spectral sequence degenerates at $E_2$, and by \cite[Proposition 2.16]{DeligneSpectral}, we have a direct sum decomposition 
\begin{equation} \label{ek} H^k(E^{n-1}) = \bigoplus_{p+q = k} H^p(\M_{1,1}, R^qf_*\qq). 
\end{equation}
By the K\"unneth formula, we have
\[
R^qf_*\qq= \bigoplus_{i_2 + \cdots + i_{n} = q}R^{i_2}\pi_*\qq \otimes \cdots \otimes R^{i_{n}}\pi_*\qq.
\]
Let $\mathbb{V}:=R^1\pi_*\qq$ and note that $R^0\pi_*\qq = \qq$ and $R^2\pi_*\qq = \qq(-1)$. 

Fix some $p$ and $(i_2, \ldots, i_{n})$ with $p + i_2 + \cdots + i_{n} = k$.
If $i_j = 2$, then we claim that \begin{equation} \label{inhk} W_kH^p(\M_{1,1},
R^{i_2}\pi_*\qq \otimes \cdots \otimes R^{i_{n}}\pi_*\qq) \subset W_kH^k(E^{n-1})
\end{equation} 
lies in the image of the first map in \eqref{wkres}.
More precisely, let
$\alpha_{j} \colon E^{n-2} \rightarrow E^{n-1}$, be the locus where the $j$th entry in $E^{n-1}$ agrees with the section $\sigma\colon \M_{1,1} \to E$. In other words, $\alpha_j$ is defined by the fiber diagram
\begin{center}
\begin{tikzcd}
E^{n-2} \arrow{d} \arrow{r}{\alpha_j} & E^{n-1} \arrow{d}{\mathrm{pr}_j} \\
\M_{1,1} \arrow{r}[swap]{\sigma} & E.
\end{tikzcd}
\end{center}
If $i_j = 2,$ we have $p + i_2 + \cdots + i_{j-1} + i_{j+1} + \cdots + i_{n} = k-2$.
Therefore, using the Leray spectral sequence for $E^{n-2} \to \M_{1,1}$, there is a corresponding term 
\begin{equation} \label{inhk-2}
\resizebox{.92\hsize}{!}{
$\begin{aligned}
W_{k-2} H^p(\M_{1,1}, R^{i_2}\pi_*\qq \otimes \cdots \otimes R^{i_{j-1}}\pi_*\qq \otimes R^{i_{j+1}}\pi_*\qq \otimes \cdots \otimes R^{i_{n}}\pi_*\qq) \subset W_{k-2} H^{k-2}(E^{n-2}).
\end{aligned}$
    }
\end{equation}
Then the pushforward $\alpha_{j*}\colon H^{k-2}(E^{n-2}) \to H^k(E^{n-1})$ sends 
the subspace on the left of \eqref{inhk-2} isomorphically onto the subspace on the left of \eqref{inhk}, which proves the claim.

It follows that $W_kH^k(\M_{1,n})$ is generated by the terms $W_k H^p(\M_{1,1}, R^{i_2} \pi_*\qq \otimes \cdots \otimes R^{i_{n}}\pi_*\qq)$ in $W_k H^k(E^{n-1})$ where all $i_j \leq 1$.
By \cite[Section~2]{Petersengenus1}, we have $W_k H^p(\M_{1,1}, \mathbb{V}^{\otimes q}) = 0$ unless $p = 1$ and $q = k-1$, in which case $W_k H^1(\M_{1,1}, \mathbb{V}^{\otimes k-1}) = \mathsf{S}_{k+1}$ by Eichler--Shimura. 

There are $\binom{n-1}{k-1}$ terms of the form $H^1(\M_{1,1}, \mathbb{V}^{\otimes k-1})$ in \eqref{ek}, coming from choosing which $k-1$ of the $n-1$ indices have $i_j =1$. Each of these terms is pulled back along the projection map $E^{n-1} \to E^{k-1}$, which remembers the $k-1$ factors for which $i_j = 1$. Let $A$ be the collection of indices $j$ such that $i_j = 1$ together with $1$. There is a commutative diagram
\begin{center}
\begin{tikzcd}
W_kH^k(E^{n-1}) \arrow{r} & W_kH^k(\M_{1,n}) \\
W_kH^k(E^{k-1}) \arrow{u} \arrow{r} & W_k H^k(\M_{1,A}). \arrow{u}
\end{tikzcd}
\end{center}
It follows that $W_kH^k(\M_{1,n})$ is generated by the pullbacks from $W_kH^k(\M_{1,A})$ as $A$ ranges over all subsets of size $k$ containing $1$.
Finally, we note that there can be no relations among these $\binom{n-1}{k-1}$ copies of $W_k H^1(\M_{1,1}, \mathbb{V}^{\otimes k-1}) = \mathsf{S}_{k+1}$ since the image of the left-hand map of \eqref{wkres} lies in the subspace of type $\mathsf{L}^i \mathsf{S}_{k+1-2i}$ for $i \geq 1$.

We have now shown that $W_kH^k(\M_{1,n}) = \mathsf{S}_{k+1} \otimes U$ for some $\binom{n-1}{k-1}$-dimensional vector space $U$. From the discussion above, it is not difficult to identify $U$ as an $\mathbb{S}_n$-representation. When 
$n = k$, the $\mathbb{S}_k$ action on $W_kH^k(\M_{1,k})$ is the sign representation.
To identify $U$ for $n > k$, let $\mathbb{S}_{n-1} \subset \mathbb{S}_n$ be the subgroup that fixes $1$. Since $W_kH^k(\M_{1,n})$ is freely generated by the pullbacks from $W_kH^k(\M_{1,A})$ as $A$ runs over subsets of size $k$ containing $1$, we have 
\[\mathrm{Res}^{\mathbb{S}_n}_{\mathbb{S}_{n-1}} U = \mathrm{Ind}_{\mathbb{S}_{k-1} \times \mathbb{S}_{n-k}}^{\mathbb{S}_{n-1}}(\mathrm{sgn}\boxtimes \mathbf{1}).\]
By the Pieri rule, we have
\begin{equation} \label{resu}
\mathrm{Ind}_{\mathbb{S}_{k-1} \times \mathbb{S}_{n-k}}^{\mathbb{S}_{n-1}}(\mathrm{sgn}\boxtimes \mathbf{1}) =
\mathrm{Ind}_{\mathbb{S}_{k-1} \times \mathbb{S}_{n-k}}^{\mathbb{S}_{n-1}}(V_{1^{k-1}} \boxtimes V_{n-k}) = V_{n-k+1,1^{k-2}} \oplus V_{n-k,1^{k-1}}.\end{equation}
By the branching rule, $V_{n-k+1,1^{k-1}}$ is the unique $\mathbb{S}_n$ representation whose restriction to $\mathbb{S}_{n-1}$ is the representation in \eqref{resu}.
\end{proof}

Since $\mathsf{S}_{14} = 0$, we have the following.
\begin{cor} \label{1ns}
$H^{13}(\Mb_{1,n})$ lies in the STE generated by $H^{11}(\Mb_{1,11})$ for all $n$.    
\end{cor}

Let $S^*_\omega$ denote the STE generated by $H^{11}(\Mb_{1,11})$. Note that $S^*_\omega$ is the STE generated by the class $\omega \in H^{11,0}(\Mb_{1,11})$ associated to the weight $12$ cusp form for $\SL_2(\zz)$, i.e., it is the smallest STE whose complexification contains $\omega$. In \cite{CanningLarsonPayne}, we showed that $S^*_\omega$ contains $H^{11}(\Mb_{g,n})$, and hence $H_{11}(\Mb_{g,n})$, for all $g$ and $n$.  In this paper, we show that it also contains $H_{13}(\Mb_{g,n})$ for all $g$ and $n$ (Theorem~\ref{thm:13}). In contrast with the system of tautological rings, an arbitrary STE need not be closed under pushforward along maps forgetting marked points or pullback to the boundary. Nevertheless, we have the following result for $S_{\omega}^*(\Mb_{g,n})$.

\begin{prop} \label{ts}
The STE $S_{\omega}^*$ is closed under the tautological operations induced by pushforward for forgetting marked points and pullback for gluing marked points.
\end{prop}
\begin{proof}
By \cite{Petersengenus1}, all even cohomology in genus $1$ is represented by boundary strata. Therefore, any product of two odd degree classes can be written as a sum of boundary strata.  It follows that every class in $S_{\omega}^*(\Mb_{g,n})$ can be represented as a linear combination of decorated graphs all of whose nontautological decorations are $H^{11}$-classes on genus 1 vertices.

The pushforward of $H^{11}(\Mb_{1,n})$ along the maps forgetting marked points is zero. Since the tautological rings are closed under pushforward, it follows that $S_{\omega}^*$ is closed under pushforward.

By \cite[Lemma 2.2]{CanningLarsonPayne}, the image of $H^{11}(\Mb_{1,n})$ under pullback to the boundary lies in $S^*_{\omega}$. By the excess intersection formula, and using the fact that $\psi$-classes in genus 1 are boundary classes, it follows that the pullback of any such decorated graph in $S_\omega^*$ to any boundary stratum is a sum of decorated graphs of the same form.  In particular, $S_{\omega}^*$ is closed under pullback to the boundary, as required.
 \end{proof}

\begin{rem} 
More generally, the Hodge groups $H^{k,0}(\Mb_{1,k})$ correspond to the space of cusp forms for $\SL_2(\zz)$ of weight $k+1$. Essentially the same argument shows that the STE generated by any subset of these cusp form spaces is closed under all of the tautological operations.
\end{rem}

\subsubsection{Genus 2}  
Here, we summarize what we need about the pure weight cohomology of $\M_{2,n}$ in low degrees, from \cite{PetersenTommasi,Petersen}. 

\begin{prop}\label{genus2even}
    Let $k\leq 10$. Then $W_{2k} H^{2k}(\M_{2,n})=RH^{2k}(\M_{2,n})$.
\end{prop}
\begin{proof}
    When $n<20$, we have $H^{2k}(\Mb_{2,n})=RH^{2k}(\Mb_{2,n})$ by \cite[Theorem 3.8]{Petersen}. Hence, $W_{2k} H^{2k}(\M_{2,n})=RH^{2k}(\M_{2,n})$ by restriction. For $n=20$ and $k\leq 10$, the same result holds (see \cite[Remark 3.10]{Petersen}).
    For $n>20$ and $k \leq 10$, all of $W_{2k} H^{2k}(\M_{2,n})$ is pulled back from $W_{2k} H^{2k}(\M_{2,m})$ where $m\leq 20$ (see Lemma \ref{blue}, below). Because the tautological ring is closed under forgetful pullbacks, the lemma follows.
\end{proof}

Applying Lemma \ref{fc} to the STE $RH^*$ immediately implies Conjecture \ref{lowk} for $g = 2$. 
\begin{cor}\label{g2lowdegreetaut}
If $k \leq 10$, then $H^{2k}(\Mb_{2,n}) = RH^{2k}(\Mb_{2,n})$ for all $n$.
\end{cor}

In odd degrees, there are restrictions on the possible motivic structures coming from the cohomology of local systems on the moduli space of principally polarized abelian surfaces.

\begin{prop} \label{danA2}
In the category of Galois representations, the pure weight cohomology of $\M_{2,n}$ in degrees 13 and 15 is of the form
\[
W_{13}H^{13}(\M_{2,n})^\semis \cong \bigoplus \mathsf{L}\mathsf{S}_{12} \mbox{ \ \ and \ \ } W_{15}H^{15}(\M_{2,n})^\semis \cong \bigoplus \mathsf{L}^2\mathsf{S}_{12}.
\]
\end{prop}
\begin{proof}
    For every $k$, the surjection
    \[
    H^k(\Mb_{2,n})\rightarrow W_{k} H^k(\M_{2,n})
    \]
    factors through $W_k H^k(\M_{2,n}^{\ct})$. Here, $\M_{2,n}^{\ct}$ is the moduli space of genus $2$ curves of compact type with $n$ markings. By \cite[Theorem 2.1(i) and (ii)]{Petersen}, 
    \[
    H^k(\M_{2,n}^{\ct})\cong \bigoplus_{p+q=k} H^p(\M_{2}^{\ct},A^q)\oplus H^p(\Sym^2 \M_{1,1}, B^q),
    \]
    where $A^q$ and $B^q$ are local systems of weight $q$, given by direct sums of Tate twists of symplectic local systems. Moreover, the terms $H^p(\Sym^2 \M_{1,1},B^q)$ map to zero under restriction to $H^k(\M_{2,n})$ by \cite[Lemma 3.3]{Petersen}.
    
    We now consider the terms $H^p(\M_{2}^{\ct},A^q)$. Let $\mathbb{V}$ be a Tate twist of a symplectic local system of weight $q$ on $\M_2^{\ct}$. If $q$ is odd, then $H^p(\M_{2}^{\ct},\mathbb{V})$ vanishes because the hyperelliptic involution acts on the fibers of $\mathbb{V}$ by $(-1)^q$. When $q$ is even, the possible Galois representations appearing in $H^p(\M_{2}^{\ct},\mathbb{V})^{\semis}$ are determined by \cite[Theorem 2.1]{Petersenlocalsystems}. When $p+q=13$ (respectively, $p+q=15)$, the only possibility of pure weight is $\mathsf{L}\mathsf{S}_{12}$ (respectively, $\mathsf{L}^2\mathsf{S}_{12})$.
\end{proof}

\begin{cor} \label{hodge13g2}
On $\Mb_{2,n}$, in the category of Galois representations, we have
\[H^{13}(\Mb_{2,n})^{\semis} \cong \bigoplus \lstw \mbox{ \ \ and \ \ }  
H^{15}(\Mb_{2,n})^{\semis} \cong  \bigoplus \l^2 \mathsf{S}_{12}.
\]
\end{cor}
\begin{proof}
For degree 13, consider the right exact sequence
 \[
    H^{11}(\tilde{\partial \M_{2,n}}) \rightarrow H^{13}(\Mb_{2,n})\rightarrow W_{13}H^{13}(\M_{2,n})\rightarrow 0.
    \]
By \cite{CanningLarsonPayne},  the semi-simplification of the left-hand side is a sum of terms $\mathsf{S}_{12}$.  Proposition \ref{danA2} shows that the semi-simplification of the right-hand side consists only of Tate twists of $\mathsf{S}_{12}$. Thus, the semi-simplification of the middle term does too.

Similarly, for degree $15$, we consider the right exact sequence
 \[
    H^{13}(\tilde{\partial \M_{2,n}}) \rightarrow H^{15}(\Mb_{2,n})\rightarrow W_{15}H^{15}(\M_{2,n})\rightarrow 0.
    \]
The boundary divisors on $\Mb_{2,n}$ are finite quotients of products of moduli spaces for genus at most $2$.
In particular, by the K\"unneth formula and vanishing of odd cohomology in degrees less than or equal to $9$, we see that the left-hand side is a sum of terms of the form $H^{13}(\Mb_{g',n'})$ or
$H^2(\Mb_{g_1,n_1}) \otimes H^{11}(\Mb_{g_2,n_2})$ where $g', g_1, g_2 \leq 2$. By the degree 13 result just proved, the semi-simplification of terms of the first kind is a sum of $\mathsf{L}\mathsf{S}_{12}$. Meanwhile, we know $H^2(\Mb_{g_1,n_1})$ is pure Tate by \cite{ArbarelloCornalba} and $H^{11}(\Mb_{g_2, n_2})^{\semis}$ is a sum of terms $\mathsf{S}_{12}$ by \cite{CanningLarsonPayne}. Hence, the semi-simplification of terms of the second kind is also a sum of $\mathsf{L}\mathsf{S}_{12}$. To conclude, note that Proposition \ref{danA2} shows that the semi-simplification of the right-hand side consists only of Tate twists of $\mathsf{S}_{12}$, so the semi-simplification of the middle term does too.
\end{proof}

\begin{rem}
    See Lemma \ref{8.1} for an analogous result in the category of Hodge structures.
\end{rem}

\section{Finite generation}\label{toprestrictions}

Let $S^*$ be an STE.  By Lemma~\ref{fc}, $S^k(\Mb_{g,n})$ contains $H^k(\Mb_{g,n})$ for all $g$ and $n$ if and only if it surjects onto $W_kH^k(\M_{g,n})/ \big(\Phi^k_{g,n} + RH^k(\M_{g,n}) \big)$ for all $g$ and $n$.  
Below we give a sufficient criterion for the vanishing of $W_kH^k(\M_{g,n})/ \Phi^k_{g,n}.$
The argument is similar to the proof of Proposition~\ref{genus1weightk}, using the K\"unneth decomposition of the Leray spectral sequence for the $n$-fold fiber product of the universal curve $\C^n \to \M_g$.

We are grateful to Dan Petersen for explaining how an earlier version of this lemma could be strengthened to the version presented here. 
This stronger version will also be useful for controlling some additional base cases in Section \ref{h13sec}.
To state it, we define the subspace
\[\Psi_{g,n}^k := \psi_1\Phi^{k-2}_{g,n} + \cdots + \psi_n \Phi^{k-2}_{g,n} \subset W_k H^k(\M_{g,n}).\]
In other words, $\Psi^k_{g,n}$ is generated by pullbacks from moduli spaces with fewer markings multiplied with $\psi$-classes. Given 
a $g$-tuple of integers $\lambda = (\lambda_1 \geq \lambda_2 \geq \cdots \geq \lambda_g \geq 0)$, let $\mathbb{V}_\lambda$ be the associated symplectic local system on $\M_g$, as in \cite{PTY}. Let $|\lambda| = \lambda_1 + \cdots + \lambda_g$, which is the weight of the local system $\mathbb{V}_\lambda$. When $|\lambda| = n$, as in the previous section, we write $V_{\lambda}$ for the irreducible $\mathbb{S}_n$ representation corresponding to $\lambda$.

\begin{lem} \label{blue} (a) If $g \geq 2$ then
\[W_kH^k(\M_{g,n})/(\Phi_{g,n}^k + \Psi_{g,n}^k) \cong \bigoplus_{|\lambda| = n} W_k H^{k - n} (\M_g, \mathbb{V}_{\lambda}) \otimes V_{\lambda^T}.\]
(b) Moreover, if $n > 0$ and $n \geq k$ then $W_kH^k(\M_{g,n}) = \Phi^k_{g,n}$.
\end{lem}
\begin{proof}
(a) Let $\pi \colon \C \to \M_g$ be the universal curve.
There is an open inclusion $\M_{g,n} \subset \C^n$, and hence restriction gives a surjection from $W_kH^k(\C^n)$ to $W_kH^k(\M_{g,n})$. Given a subset $A \subset \{1, \ldots, n\}$, let $\C^n \to \C^A$ be the projection onto the factors indexed by $A$. We consider the following three subspaces of $H^k(\C^n)$:

 \begin{enumerate}
     \item[(i)] $\tilde{\Phi}$, the span of the pullbacks of $H^k(\C^{\{i\}^c})$ along projection $\C^n \to \C^{\{i\}^c}$
     \item[(ii)] $\tilde{\Psi}$, the span of $\psi_i \cdot H^{k-2}(\C^{\{i\}^c})$
     \item[(iii)] $\tilde{\Delta}$, the span of $\Delta_{ij} \cdot H^{k-2}(\C^{\{i, j\}^c})$, where $\Delta_{ij}$ denotes the pullback of the class of the diagonal in $H^2(\C^{\{i,j\}})$.
 \end{enumerate}
Note that $\tilde{\Delta}$ lies in the kernel of $H^k(\C^n) \to H^k(\M_{g,n})$, and
the images of the weight $k$ parts of $\tilde{\Phi}$ and $\tilde{\Psi}$ under $W_kH^k(\C^n) \to W_kH^k(\M_{g,n})$ are the subspaces $\Phi^k_{g,n}$ and $\Psi^k_{g,n}$ respectively.
To prove part (a), it thus suffices to show that
\begin{equation} \label{fora} H^k(\C^n)/(\tilde{\Phi} + \tilde{\Psi} + \tilde{\Delta}) \cong \bigoplus_{|\lambda|=n} H^{k-n}(\M_g, \mathbb{V}_{\lambda}) \otimes V_{\lambda^T}.
\end{equation}
We explain how \eqref{fora} follows from \cite{PTY}.

Since $f\colon \C^n \to \M_g$ is smooth and proper, the Leray spectral sequence degenerates at $E_2$, and by \cite[Proposition 2.16]{DeligneSpectral}, we have a direct sum decomposition 
\[H^k(\C^n) = \bigoplus_{p+q = k}H^p(\M_g, R^q f_*\qq). \]
Applying the K\"unneth formula, we obtain
\begin{equation} \label{kde} H^k(\C^n) = \bigoplus_{\substack{p + q = k \\i_1 + \ldots+ i_n = q }} H^p\left(\M_g, R^{i_1} \pi_* \qq \otimes \cdots \otimes R^{i_n} \pi_* \qq \right).
\end{equation}
As observed in \cite[Section 5.2.2]{PTY}, the subspace $\widetilde{\Phi}$ corresponds to the span of the summands where some $i_s = 0$. 
Meanwhile, modulo $\tilde{\Phi}$, the subspace $\tilde{\Psi}$ corresponds to the span of summands where some $i_s = 2$. This follows from the formulas for the projector $\pi_2$ (which projects onto such summands) in \cite[Section 5.1]{PTY}. Quotienting by $\tilde{\Phi} + \tilde{\Psi}$ thus leaves the term with $i_1 = \ldots = i_n = 1$. The local system $(R^1\pi_*\qq)^{\otimes n}$ corresponds to the $n$th tensor power of the standard representation of $\mathrm{Sp}_{2g}$. It thus decomposes into a direct sum of irreducible local systems as in \cite[Section 3.2]{PTY}.
As explained there, we have a natural projection
\begin{equation} \label{pl} (R^1\pi_*\qq)^{\otimes n} \to \bigoplus_{|\lambda| = n} \mathbb{V}_{\lambda} \otimes V_{\lambda^T}
\end{equation}
whose kernel is spanned by the image of the $\binom{n}{2}$ insertion maps
\[(R^1 \pi_* \qq)^{\otimes n -2} \to (R^1 \pi_* \qq)^{\otimes n}\]
given by inserting the class of the symplectic form. Modulo other K\"unneth components, the class of the diagonal $\Delta_{i,j}$ is the class of the symplectic form in $H^{1}(\C) \otimes H^1(\C) \subset H^2(\C^{\{i,j\}})$. Summarizing, we have found
\[H^k(\C^n)/(\widetilde{\Phi} + \widetilde{\Psi}) \cong H^{k-n}(\M_g, (R^1\pi_*\qq)^{\otimes n}) \to \bigoplus_{|\lambda| =n } H^{k-n}(\mathbb{V}_{\lambda}) \otimes V_{\lambda^T}, \]
and the kernel is spanned by $\tilde{\Delta}$. This establishes \eqref{fora}

(b) Now assume $n \geq k$. If some $i_s = 2$ and $i_1 + \cdots + i_n = k$, we must have some other $i_{s'} = 0$. It follows that
$\tilde{\Psi} \subset \tilde{\Phi}$ and hence $\Psi^k_{g,n} \subset \Phi^k_{g,n}$. Thus it suffices to see that the right-hand side of (a) vanishes when $n \geq k$ and $n > 0$. This is immediate when $n > k$. Meanwhile, for the case $n =k$, we have $H^0(\M_g, \mathbb{V}_{\lambda}) = 0$ since $\mathbb{V}_{\lambda}$ is nontrivial.
\end{proof}

We note a neat consequence of Lemma \ref{blue}(a) in genus $2$, which was also pointed out to us by Dan Petersen.

\begin{lem} \label{odd}
Suppose $n$ is odd. The STE generated by $H^k(\Mb_{2,n-1}), H^{k-2}(\Mb_{2,n-1})$
and $H^{k-2}(\Mb_{1,m})$ for $m \leq n+2$ contains $H^k(\Mb_{2,n})$.
\end{lem}
\begin{proof}
The STE generated by $H^k(\Mb_{2,n-1}), H^{k-2}(\Mb_{2,n-1})$
and $H^{k-2}(\Mb_{1,m})$ for $m \leq n+2$ contains all classes in $H^{k}(\Mb_{2,n})$ that are pushed forward from the boundary, pulled back from less marked points, or products of $\psi$ classes with classes pulled back from less marked points. It therefore suffices to show that $W_kH^k(\M_{2,n})/(\Phi^k_{2,n} + \Psi^k_{2,n}) = 0$.
By Lemma \ref{blue}(a), this quotient is a sum of the pure cohomology of local systems of weight $n$. But on $\M_2$, the hyperelliptic involution acts on local systems of weight $n$ by $(-1)^n$, so local systems of odd weight have no cohomology.
\end{proof}

\begin{figure}[h!]
    \centering
    \includegraphics[width=3in]{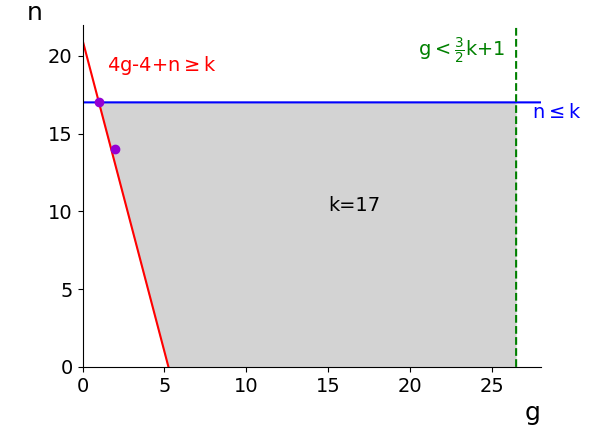}
    \caption{The argument in the proof of Theorem~\ref{thm:fgSTE} shows that $W_{17}H^{17}(\M_{g,n})/(\Phi^{17}_{g,n} + RH^{17}(\M_{g,n}))$ vanishes for $(g,n)$ outside the gray shaded region.  Note that this quotient does \emph{not} vanish for $(g,n)$ equal to $(1,17)$ and $(2,14)$, which are pictured by purple dots.}
    \label{fig:my_label}
\end{figure}

\begin{proof}[Proof of Theorem \ref{thm:fgSTE}]
Let $S^*$ be the STE generated by the cohomology groups listed in the statement of the theorem.
By Lemma \ref{fc}, it suffices to check that $S^*(\Mb_{g',n'})$ surjects onto $W_{k'}H^{k'}(\M_{g',n'})/\left(\Phi^{k'}_{g',n'} + RH^{k'}(\M_{g',n'})\right)$ for all $(g', n', k')$ with $k' \leq k-2$. 
By Lemma \ref{blue}(b), the target vanishes when $n' > k'$. 
By the vcd of $\M_{g',n'}$, it also vanishes when $k' > 4g' - 4 + n'$. Finally, $H^{k'}(\M_{g',n'})$ is tautological for  
$k' \leq \frac{2g'-2}{3}$ \cite{Wahl}.
\end{proof}

\begin{rem}
One can prove analogues of Theorem \ref{thm:fgSTE} and Lemma \ref{fc}  for homology. In particular, for each $k$, there is a finitely generated STE that contains $H_k(\Mb_{g,n})$ for all $g$ and $n$, with  explicit bounds on the generators. 
\end{rem}

\section{The Chow--K\"unneth generation Property} \label{agp}

In this section we prove a key lemma about the cycle class map for spaces that have the following property.

\begin{definition}
Let $X$ be a smooth algebraic stack of finite type over a field, stratified by quotient stacks. We say that $X$ has the \emph{Chow--K\"unneth generation Property} (CKgP) if for all algebraic stacks $Y$ (of finite type, stratified by quotient stacks) the exterior product
\[
A_*(X)\otimes A_*(Y)\rightarrow A_*(X\times Y)
\]
is surjective. 
\end{definition}

For convenience, we record here several properties of the CKgP, all of which are proven in \cite[Section 3.1]{CL-CKgP}.
\begin{prop}\label{CKgPprops}
Let $X$ be a smooth algebraic stack of finite type over a field, stratified by quotient stacks. 
    \begin{enumerate}
        \item \label{open} if $U\subset X$ is open and $X$ has the CKgP, then $U$ has the CKgP;
        \item \label{proper} if $Y\rightarrow X$ is proper, surjective, representable by DM stacks, and $Y$ has the CKgP, then $X$ has the CKgP;
        \item \label{stratification} if $X$ admits a finite stratification $X=\coprod_{S\in \mathcal{S}} S$ such that each $S$ has the CKgP, then $X$ has the CKgP;
\item \label{affine} if $V\rightarrow X$ is an affine bundle, then $V$ has the CKgP if and only if $X$ has the CKgP;    
        \item \label{Grassmann} if $X$ has the CKgP, and $G\rightarrow X$ is a Grassmann bundle, then $G$ has the CKgP;
        \item \label{BGL} if $X = \BGL_n, \BSL_n$,  or $\BPGL_n$, then $X$ has the CKgP.
    \end{enumerate}
\end{prop}

If $X$ is smooth and proper and has the CKgP, then the cycle class map for $X$ is an isomorphism \cite[Lemma 3.11]{CL-CKgP}. 

When $X$ is smooth but not necessarily proper and has the CKgP, the cycle class map is not necessarily an isomorphism. Nevertheless, we have the following useful substitute (cf. \cite{Laterveer, Totarolinear} for slightly different statements with similar proofs).

\begin{lem}\label{cycleclass}
Let $X$ be an open substack of a smooth proper Deligne--Mumford stack $\Xb$ over the complex numbers. If $X$ has the CKgP, then the cycle class map
\[
\cl\colon\bigoplus_i A_i(X)\rightarrow \bigoplus_k W_k H^k(X)
\]
is surjective. In particular, if $k$ is odd then 
$W_kH^k(X)=0$, and if $k$ is even then $W_k H^k(X)$ is pure Hodge--Tate. 
\end{lem}
\begin{proof}
Set $d:=\dim \Xb$ and $D:=\Xb\smallsetminus X$. Let $\Delta\subset \Xb\times \Xb$ denote the diagonal.
Because $X$ has the CKgP, the exterior product map
\[
\bigoplus_{\ell=0}^{d} A^\ell(\Xb)\otimes A^{d-\ell}(X)\rightarrow A^d(\Xb\times X)
\]
is surjective. We have the excision exact sequence
\[
A^d(\Xb\times D)\rightarrow A^d(\Xb\times \Xb)\rightarrow A^d(\Xb\times X)\rightarrow 0.
\]
It follows that we can write the class of the diagonal in $A^d(\Xb \times \Xb)$ as
\begin{equation}\label{decomposition}
\Delta = \Gamma + \Delta^0 + \Delta^1 +\cdots + \Delta^d,
\end{equation}
where $\Gamma$ is supported on $\Xb\times D$ and each $\Delta^\ell$ is a linear combination of cycles of the form $V_i^\ell\times W_i^{d-\ell}$, where $V_i^{\ell}$ and $W_i^{d-\ell}$ are subvarieties of $\Xb$ of codimension $\ell$ and $d-\ell$, respectively. 

Let $p_1$ and $p_2$ be the projections of $\overline{X} \times \overline{X}$ to the first and second factors respectively. Given a class $\Psi \in H^*(\overline{X} \times \overline{X})$, we write $\Psi_*\colon H^*(\overline{X}) \to H^*(\overline{X})$ for the associated correspondence, defined by $\Psi_*\alpha = p_{2*}(p_1^*\alpha\cdot \Psi)$.

Let $a\in W_k H^k(X)$. Let $\alpha$ be a lift of $a$ in $W_k H^k(\Xb)=H^k(\Xb)$. Then
\[
\alpha = \Delta_*\alpha = (\Gamma_* + \Delta^0_* + \Delta^1_* +\cdots + \Delta^d_*)\alpha.
\]
First, we study $\Delta^\ell_*\alpha = p_{2*}(p_1^*\alpha\cdot \Delta^{\ell}).$
 Note that $p_1^*\alpha\cdot \Delta^{\ell}$ vanishes for dimension reasons if $k+2\ell>2d$. Furthermore, since $p_2$ is of relative dimension $d$, the pushforward by $p_2$ of any cycle will vanish if $k+2\ell<2d$. Thus, the only nonzero terms occur when $k=2d-2\ell$. Moreover, $\Delta^{\ell}_*\alpha$ is a linear combination of the form $\sum c_i W_i^{d-\ell}$, so it lies in the image of the cycle class map.

Next, we study $\Gamma_*\alpha$. It suffices to treat the case that $\Gamma$ is the class of a subvariety of $\Xb\times \Xb$ contained in $\Xb\times D$, as it is a linear combination of such subvarieties. In this case, the map $p_2|_{\Gamma}\colon\Gamma\rightarrow \Xb$ factors through $D\rightarrow \Xb$. Thus, $\Gamma_*\alpha$ maps to zero under the restriction to $W_k H^k(X)$ because the correspondence map factors through the cohomology of the boundary $D$.
Thus, $\alpha|_{X}=a$ is in the image of the cycle class map. 
\end{proof}

\begin{rem}
Essentially the same argument (only with $\qq_\ell$-coefficients) gives a similar statement for the cycle class map to $\ell$-adic \'etale cohomology over an arbitrary field.
\end{rem}

The first two authors have previously given many examples of moduli spaces $\M_{g,n}$ that have the CKgP and also satisfy $A^*(\M_{g,n}) = R^*(\M_{g,n})$ \cite[Theorem 1.4]{CL-CKgP}.  For the inductive arguments in this paper, we need more base cases in genus $7$. In the next section, we prove that $\M_{7,n}$ has the CKgP and $A^*(\M_{7,n}) = R^*(\M_{7,n})$, for $n \leq 3$. The table below records the previously known results from \cite[Theorem 1.4]{CL-CKgP} together with Theorem~\ref{thm:genus7}.

\begin{table}[h!]
\centering
\begin{tabular}{|c|c|c|c|c|c|c|c|c|} 
 \hline
 $g$ & $0$ & $1$ & $2$ & $3$ & $4$ & $5$ & $6$ & $7$ \\
 \hline
 $c(g)$ & $\infty$ & $10$ & $10$ & $11$ & $11$ & $7$ & $5$ & $3$  \\ 
\hline
\end{tabular}
\vspace{.1in}
\caption{$\M_{g,n}$ has the CKgP and $A^*(\M_{g,n})=R^*(\M_{g,n})$, for $n \leq c(g)$, by \cite[Theorem 1.4]{CL-CKgP} and Theorem~\ref{thm:genus7}.}
\label{ckgptable}
\end{table}
\noindent 
The following is a consequence of Lemma \ref{cycleclass}.
\begin{prop}\label{openCKgPresults}
For all $g\leq 7$ and $n \leq c(g)$ as specified in Table~\ref{ckgptable}, we have
\[
W_k H^k(\M_{g,n})=RH^k(\M_{g,n}).
\]
In particular, if $k$ is odd and $n \leq c(g)$, then
\[
W_k H^k(\M_{g,n}) = \mathrm{gr}^W_{k} H^{k}_c(\M_{g,n})=0.
\]
\end{prop}

\begin{proof}[Proof of Theorem~\ref{thm:g3}, assuming Theorem~\ref{thm:131415}]
By the Hard Lefschetz theorem, it suffices to prove that $H^k(\Mb_{3,n})^{\semis}$ is a polynomial in $\mathsf{L}$ and $\mathsf{S}_{12}$ for fixed $k\leq 17$ and $n\leq 11$. Consider the right exact sequence
\[
H^{k-2}(\tilde{\partial \M_{3,n}})\rightarrow H^k(\Mb_{3,n})\rightarrow W_k H^k(\M_{3,n})\rightarrow 0.
\]
By Theorem \ref{thm:131415} and \cite[Theorems 1.1, 1.2 and Lemma 2.1]{CanningLarsonPayne}, for $k'\leq 15$, $H^{k'}(\Mb_{g',n'})^{\semis}$ is a polynomial in $\mathsf{L}$ and $\mathsf{S}_{12}$ for all $g'\neq 1$ and all $n'$ as well as for $g'=1$ and $n'\leq 13$.
Since we assume $n \leq 11$, no $\Mb_{1,n'}$ with $n' \geq 14$ appears as a factor in a component of the normalized boundary.
It follows that, $H^{k-2}(\tilde{\partial \M_{3,n}})^{\semis}$ is a polynomial in $\mathsf{L}$ and $\mathsf{S}_{12}$.
By \cite[Theorem 1.4]{CL-CKgP}, $\M_{3,n}$ has the CKgP and $A^*(\M_{3,n}) = R^*(\M_{3,n})$ for $n \leq 11$. Thus $W_k H^k(\M_{3,n})=RH^k(\M_{3,n})$ by Lemma \ref{cycleclass}. Hence, $W_k H^k(\M_{3,n})$ is pure Hodge--Tate, and the theorem follows.
\end{proof}
\section{The CKgP in genus 7 with at most 3 marked points}\label{sec:7CKgP}

In order to prove Theorems \ref{thm:131415}, \ref{thm:taut}\eqref{it:homology} for $k = 14$, and \ref{thm:15}, we need more base cases in genus $7$. The required base cases are given by Theorem~\ref{thm:genus7}, which we now prove. 

In order to prove Theorem~\ref{thm:genus7}, we filter  $\M_{7,n}$ by gonality:
\[
\M^{2}_{7,n}\subset \M^{3}_{7,n}\subset \M^4_{7,n} \subset \M^{5}_{7,n}. 
\]
Here, $\M^{k}_{7,n}$ is the locus parametrizing smooth curves $C$ with $n$ marked points that admit a map of degree at most $k$ to $\p^1$. Standard results from Brill--Noether theory show that the maximal gonality of a genus $7$ curve is $5$. By Proposition \ref{CKgPprops}\eqref{stratification}, to show that $\M_{7,n}$ has the CKgP, it suffices to show that each gonality stratum
\[
\M^k_{7,n}\smallsetminus \M^{k-1}_{7,n}.
\]
has the CKgP. Moreover, to show that $A^*(\M_{7,n})=R^*(\M_{7,n})$, it suffices to show for each $k$ that all classes supported on $\M^k_{7,n}$ are tautological up to classes
supported on $\M^{k-1}_{7,n}$. In other words, we must show that every class in $A^*(\M_{7,n}^k \smallsetminus \M_{7,n}^{k-1})$
pushes forward to a class in $A^*(\M_{7,n} \smallsetminus \M^{k-1}_
{7,n})$ that is the restriction of a tautological class
on $\M_{7,n}$. 

\subsection{Hyperelliptic and trigonal loci}

By \cite[Lemma 9.9]{CL-CKgP}, if $n\leq 14$, then $\M^3_{7,n}$ has the CKgP and all classes in $A^*(\M_{7,n})$ supported on $\M^3_{7,n}$ are tautological. Note that this includes the hyperelliptic locus.

\subsection{The tetragonal locus}
To study the tetragonal locus $\M^4_{7,n}\smallsetminus \M^3_{7,n}$, we will use the Hurwitz stack $\H_{4,g,n}$ parametrizing degree $4$ covers $f\colon C\rightarrow \p^1$, where $C$ is a smooth curve of genus $g$ with $n$ marked points.
There is a forgetful morphism
    \[
    \beta_n\colon\H_{4,g,n}\rightarrow \M_{g,n}.
    \]
Restricting to curves of gonality exactly $4$, we obtain a proper morphism
    \[
    \beta_n'\colon\H_{4,g,n}^{\Diamond}:= \H_{4,g,n}\smallsetminus \beta_n^{-1}(\M^3_{g,n})\rightarrow \M_{g,n}\smallsetminus \M^3_{g,n}
    \]
with image $\M_{g,n}^4 \smallsetminus \M_{g,n}^3$. To show that $\M^4_{7,n}\smallsetminus \M^3_{7,n}$ has the CKgP, it suffices to show that $\H_{4,7,n}^\Diamond$ has the CKgP by Proposition \ref{CKgPprops}\eqref{proper}. We will do so by further stratifying $\H_{4,7,n}^\Diamond$.

In \cite[Section 4.4]{CanningLarson789}, the first two authors studied a stratification of $\H_{4,7} :=  \H_{4,7,0}$ with no markings. Here we carry out a similar analysis with marked points. The Casnati--Ekedahl structure theorem \cite{CasnatiEkedahl} associates to a point in $\H_{4,g}$ a rank $3$ vector bundle $E$ and a rank $2$ vector bundle $F$ on $\p^1$, both of degree $g+3$, equipped with a canonical isomorphism $\det E \cong \det F$ \cite[Section 3]{part1}. Let $\B$ be the moduli stack of pairs of vector bundles $(E, F)$ on $\pp^1$ of degree $g+3$, together with an isomorphism of their determinants as in \cite[Definition 5.2]{part1}. Let $\pi \colon  \P \to \B$ be the universal $\pp^1$-fibration, and let $\E$ and $\F$ be the universal bundles on $\P$. There is a natural morphism $\H_{4,g} \to \B$ that sends a degree $4$ cover to its associated pair of vector bundles. 
Moreover, the Casnati--Ekedahl construction gives an embedding of the universal curve $\C$ over $\H_{4,g}$ into $\p \E^{\vee}$.

Consider the natural commutative diagram
\begin{equation} \label{comd}
\begin{tikzcd}
\C_n \arrow{r}{b} \arrow{d} & \C  \arrow{d} \arrow{r}{a} & \pp \E^\vee \arrow{d} \\
\arrow[bend left = 50, "\sigma_i"]{u} \H_{4,g,n} \arrow{r} & \H_{4,g} \arrow{r} & \B,
\end{tikzcd}
\end{equation} 
where $\C_n$ is the universal curve over $\H_{4,g,n}$. For each $i$, the map $a \circ b \circ \sigma_i$ sends a pointed curve to the image of the $i$th marking under the Casnati--Ekedahl embedding. Taking the product of these maps for $i = 1, \ldots, n$, we obtain a commutative diagram
\begin{equation} \label{HB}
\begin{tikzcd}
\H_{4,g,n}^{\Diamond}\arrow{r} \arrow{d} & \H_{4,g,n} \arrow{r} \arrow{d} & (\pp \E^\vee)^n \arrow{d} \\
\H_{4,g}^{\Diamond}\arrow{r} & \H_{4,g} \arrow{r} & \B.
\end{tikzcd}
\end{equation}

\begin{lem}[Lemma 10.5 of \cite{CL-CKgP}]\label{tot}
Suppose $x \in A^*(\H_{4,g,n}^\Diamond)$ lies in the image of the map $A^*((\pp \E^\vee)^n) \to A^*(\H_{4,g,n}^\Diamond)$. Then $\beta'_{n*}x$ is tautological on $\M_{g,n} \smallsetminus \M_{g,n}^3$.
\end{lem}

The splitting types of the Casnati--Ekedahl bundles $E$ and $F$ induce a stratification on $\H_{4,7}$. 
We write $E = (e_1, e_2, e_3)$ and $F = (f_1, f_2)$ to indicate that the bundles have splitting types
\[E=\O(e_1)\oplus \O(e_2)\oplus \O(e_3) \qquad \text{and} \qquad F=\O(f_1)\oplus \O(f_2).\]
We will consider a stratification into three pieces 
\[
\H_{4,7,n}=X_n\sqcup Y_n \sqcup Z_n.
\]
The three strata correspond to unions of splitting types of $E$ and $F$. The possible splitting types are recorded in \cite[Section 4.4]{CanningLarson789}.
The locus $Z_n$ is the set of covers with maximally unbalanced splitting types and parametrizes hyperelliptic curves \cite[Equation 4.5]{CanningLarson789}. Its image in $\M_{7,n}$ is contained in $\M_{7,n}^3$, which has the CKgP when $n\leq 14$, as noted above. We will show that $X_n$ and $Y_n$ have the CKgP and that $A^*((\p \E^{\vee})^n)$ surjects onto the Chow ring of their union, which is  $\H_{4,7,n}^\Diamond$. 
We start with $X_n$.

\medskip

Let $X_n\subset \H_{4,7,n}$ denote the locus of covers with splitting types $E = (3, 3, 4)$ and $F = (5,5)$, or $E = (3, 3, 4)$ and $F = (4, 6)$.

\begin{lem} \label{goodopen}
If $n \leq 3$, then $X_n$ has the CKgP and $A^*((\pp \E^\vee)^n) \to A^*(X_n)$ is surjective.
\end{lem}
\begin{proof}
Taking $f =4$ in \cite[Definition 10.8]{CL-CKgP}, we have $X_n = \H_{4,7,n}^4$. The result then follows from \cite[Lemmas 10.11 and 10.12]{CL-CKgP} with $g = 7$ and $f = f_1 = 4$.
\end{proof}

Now let $Y_n\subset \H_{4,7,n}$ be the union 
\[
Y_n = \Sigma_{2,n}\sqcup \Sigma_{3,n},
\]
where $\Sigma_{2,n}$ parametrizes covers with splitting types $E = (2, 4, 4)$ and $F = (4,6)$, and $\Sigma_{3,n}$ parametrizes covers with splitting types $E = (2, 3, 5)$ and $F = (4,6)$.

Recall that $\pi\colon\P \rightarrow \B$ is the structure map for the $\pp^1$ bundle $\P$. Let $\gamma\colon \pp \E^\vee \to \P$ denote the structure map and $\eta_i\colon (\pp \E^\vee)^n \to \pp \E^\vee$ denote the $i$th projection. Define
\[
z_i := \eta_i^* \gamma^* c_1(\O_{\P}(1)) \mbox{ \ \  and \ \ } \zeta_i := \eta_i^* c_1(\O_{\pp \E^\vee}(1)).\] Then $z_i$ and $\zeta_i$ generate $A^*((\pp \E^\vee)^n)$ as an algebra over $A^*(\B)$.  Write $\Sigma_{\ell} := \Sigma_{\ell, 0}$. 

\begin{lem} \label{s1}
For $\ell = 2,3$ and $n \leq 3$, there is a surjection 
\[A^*(\Sigma_{\ell})[z_1, \ldots, z_n, \zeta_1, \ldots, \zeta_n] \to A^*(\Sigma_{\ell,n})\]
induced by $\Sigma_{\ell,n}\rightarrow \Sigma_\ell$ and restriction from $(\p \E^{\vee})^n$.
Moreover, $\Sigma_{\ell,n}$ has the CKgP.
\end{lem}
\begin{proof}
Let $\vec{e}$ and $\vec{f}$ be the splitting types associated to $\Sigma_{\ell}$ and let $\B_{\vec{e}, \vec{f}} \subset \B$ be the locally closed substack that parametrizes pairs of bundles $(E, F)$ with locally constant splitting types $\vec{e}$ and $\vec{f}$. We write $\O(\vec{e}) := \O(e_1) \oplus \cdots \oplus \O(e_k)$. By construction, $\Sigma_{\ell,n}$ is the preimage of $\B_{\vec{e}, \vec{f}}$ along $\H_{4,7,n} \to \B$.
We therefore study the base change of \eqref{HB} along $\B_{\vec{e}, \vec{f}} \to \B$:
\begin{equation}
\begin{tikzcd}
\Sigma_{\ell,n} \arrow{r} \arrow{d} & (\pp \E^\vee|_{\B_{\vec{e},\vec{f}}})^n \arrow{d} \\
\Sigma_{\ell} \arrow{r} & \B_{\vec{e},\vec{f}}.
\end{tikzcd}
\end{equation}

We now recall the description of $\Sigma_\ell$ as an open substack of a vector bundle on $\B_{\vec{e}, \vec{f}}$, as in \cite[Lemma 3.10]{CanningLarson789}. Let
\[U \subset H^0(\pp^1,\O(\vec{f})^\vee \otimes \Sym^2 \O(\vec{e})) = H^0(\pp \O(\vec{e})^\vee, \gamma^*\O(\vec{f})^\vee \otimes \O_{\pp \O(\vec{e})^\vee}(2))\]
be the open subset of equations that define a smooth curve, as in \cite[Lemma 3.10]{CanningLarson789}. Then 
\[\Sigma_\ell = [(U \times \gg_m)/\SL_2 \ltimes (\Aut(\O(\vec{e})) \times \Aut(\O(\vec{f})))].\]
The stack $\B_{\vec{e}, \vec{f}}$ is the part obtained by forgetting $U$:
\[\B_{\vec{e}, \vec{f}} = [\gg_m/\SL_2 \ltimes (\Aut(\O(\vec{e})) \times \Aut(\O(\vec{f})))].\]
As explained in \cite[Equation 3.1]{CanningLarson789}, there is a product of stacks $\BGL_n$ which is an affine bundle over
$\mathrm{B}\!\Aut(\O(\vec{e}))$. As such, $\B_{\vec{e}, \vec{f}}$ has the CKgP by Proposition \ref{CKgPprops}\eqref{affine} and \eqref{BGL}. It follows that $\Sigma_{\ell}$ also has the CKgP by Proposition \ref{CKgPprops}\eqref{open} and \eqref{affine}.

Let us define the rank $2$ vector bundle $\W := \gamma^*\F^\vee \otimes \O_{\pp \E^\vee}(2)$ on $\pp \E^\vee$. Write $\W_{\vec{e},\vec{f}}$ for the restriction of $\W$ to $\pp \E^\vee|_{\B_{\vec{e},\vec{f}}}$. The discussion above says that $\Sigma_{\ell}$ is an open substack of the vector bundle $(\pi \circ \gamma)_*\W_{\vec{e},\vec{f}}$ on $\B_{\vec{e},\vec{f}}$.

Next, we give a similar description with marked points. Consider the evaluation map  
\[
    (\pi\circ\gamma)^*(\pi\circ\gamma)_*\W_{\vec{e},\vec{f}} \rightarrow \W_{\vec{e},\vec{f}}.
    \]
    Pulling back to the fiber product $(\pp \E^\vee|_{\B_{\vec{e},\vec{f}}})^n$, we obtain 
    \begin{equation}\label{eval}
    \eta_j^*(\pi\circ\gamma)^*(\pi\circ\gamma)_*\W_{\vec{e},\vec{f}}\rightarrow \bigoplus_{i=1}^n\eta_i^*\W_{\vec{e},\vec{f}}.
    \end{equation}
    Note that $\eta_j^*(\pi\circ\gamma)^*(\pi\circ\gamma)_*\W_{\vec{e},\vec{f}}$ is independent of $j$. 
The kernel $\Y$ of \eqref{eval} parametrizes tuples of $n$ points on $\pp \O(\vec{e})^\vee$ together with a section of 
\[H^0(\pp \O(\vec{e})^\vee, \gamma^* \O(\vec{f})^\vee \otimes \O_{\pp \O(\vec{e})^\vee}(2))\]
that vanishes on the points. There is therefore a natural map $\Sigma_{\ell,n} \hookrightarrow \Y$ defined by sending a pointed curve to the images of the points on $\pp \O(\vec{e})^\vee$ and the defining section of the curve. The kernel $\Y$ is not locally free. Nevertheless, its restriction to the open substack $U_n \subset (\pp \E^\vee|_{\B_{\vec{e},\vec{f}}})^n$ where \eqref{eval} is surjective is locally free.

We claim that the map $\Sigma_{\ell,n} \to \Y$ factors through $\Y|_{U_n}$. To see this, 
suppose $C \subset \pp \O(\vec{e})^\vee$ is the vanishing of a section of $\gamma^* \O(\vec{f})^\vee \otimes \O_{\pp \O(\vec{e})^\vee}(2)$ and $C$ is smooth and irreducible.
By \cite[Lemma 10.6]{CL-CKgP}, if $n\leq 3$, the evaluation map \eqref{eval} is surjective at any tuple of $n$ distinct points on $C$. It follows that the image of the composition $\Sigma_{\ell,n} \rightarrow \Y \rightarrow (\pp \E^\vee|_{\B_{\vec{e},\vec{f}}})^n$ is contained in $U_n$. Hence, $\Sigma_{\ell,n} \hookrightarrow \Y|_{U_n}$.

In summary, we have a sequence of maps
\[\Sigma_{\ell,n} \hookrightarrow \Y|_{U_n} \rightarrow U_n \hookrightarrow (\pp \E^\vee|_{\B_{\vec{e},\vec{f}}})^n \rightarrow (\P|_{\B_{\vec{e},\vec{f}}})^n \rightarrow  \B_{\vec{e},\vec{f}}. \]
each of which is an open inclusion, vector bundle, or product of projective bundles. Since $\B_{\vec{e},\vec{f}}$ has the CKgP, it follows that $\Sigma_{\ell,n}$ also has the CKgP by Proposition \ref{CKgPprops}\eqref{open},\eqref{affine}, and \eqref{Grassmann}. Moreover, we see that
\[A^*(\B_{\vec{e},\vec{f}})[z_1, \ldots, z_n, \zeta_1, \ldots, \zeta_n] \to A^*(\Sigma_{\ell,n})\]
is surjective. Finally note that $\Sigma_{\ell,n} \to \B_{\vec{e},\vec{f}}$ factors through $\Sigma_\ell$, so $A^*(\B_{\vec{e},\vec{f}}) \to A^*(\Sigma_{\ell,n})$ factors through $A^*(\Sigma_{\ell})$. This proves the claim.
\end{proof}

\begin{cor}
For $n \leq 3$, $\H_{4,7,n}^\Diamond$ has the CKgP. Hence, $\M^4_{7,n}$ has the CKgP.
\end{cor}
\begin{proof}
We have $\H_{4,7,n}^\Diamond = X_n \sqcup Y_n$. By Lemmas \ref{goodopen} and \ref{s1}, each of these pieces has the CKgP. Note that $\H^{\Diamond}_{4,7,n}$ maps properly onto $\M^4_{7,n}\smallsetminus \M^3_{7,n}$ and $\M_{7,n}^3$ has the CKgP. Thus the result follows by Proposition \ref{CKgPprops}\eqref{proper}-\eqref{stratification}.
\end{proof}

The next step is to show that all classes in $A^*(\Sigma_{\ell,n})$ are restrictions from $(\p \E^{\vee})^n$.
\begin{lem} \label{res}
For $\ell = 2,3$ and $n \leq 3$, the restriction $A^*((\p \E^{\vee})^n) \to A^*(\Sigma_{\ell,n})$ is surjective.
\end{lem}
\begin{proof}
Consider the following diagram:
\begin{equation}
\begin{tikzcd}
A^*((\pp \E^\vee)^n)  \arrow{r} & A^*(\Sigma_{\ell,n})\\
A^*(\B)[z_1, \ldots, z_n, \zeta_1, \ldots, \zeta_n] \arrow{u} \arrow{r} & A^*(\Sigma_{\ell})[z_1, \ldots, z_n, \zeta_1, \ldots, \zeta_n]. \arrow{u}
\end{tikzcd}
\end{equation}
The map $A^*(\B) \to A^*(\Sigma_{\ell})$ is surjective  
by \cite[Lemma 4.2]{CanningLarson789} for $\ell =2$, and by \cite[Lemma 4.3(1)]{CanningLarson789} for $\ell = 3$. 
Therefore, the bottom horizontal arrow is surjective. By Lemma \ref{s1}, the right vertical arrow is also surjective. Hence, the top horizontal arrow is surjective.
\end{proof}

\begin{lem} \label{all}
For $n \leq 3$, the pullback map $A^*((\pp \E^\vee)^n) \to A^*(\H_{4,7,n}^\Diamond)$
is surjective.
\end{lem}
\begin{proof}
First, we fix some notation. Let
\[
\Sigma_{3,n}\xhookrightarrow{\jmath}Y_n\xhookrightarrow{\iota} \H_{4,7,n}^{\Diamond}
\]
denote the natural closed inclusion maps. Let $\phi \colon \H_{4,7,n}^\Diamond \to (\pp \E^\vee)^n$ be the map from the top left to top right in diagram \eqref{HB}.
We let $\phi':X_n\rightarrow (\p \E^{\vee})^n$ be the composite of the open inclusion $X_n \hookrightarrow \H_{4,7,n}^{\Diamond}$ and $\phi$. Let $\psi:=\phi\circ \iota$, and let $\psi'$ be the composite of the open inclusion $\Sigma_{2,n}\hookrightarrow Y_n$ and $\psi$.

Consider the following commutative diagram, where the bottom row is exact.
\begin{equation} \label{md}
\begin{tikzcd}
& A^*((\pp \E^\vee)^n)  \arrow{d}[swap]{\phi^*} \arrow{dr}{\phi'^*} & \\
A^{*-2}(Y_n) \arrow{r}{\iota_*} & A^*(\H_{4,7,n}^{\Diamond}) \arrow{r} & A^*(X_n) \arrow{r} & 0.
\end{tikzcd}
\end{equation}
By Lemma \ref{goodopen}, $\phi'^*$ is surjective.
It thus suffices to show that the image of $\iota_*$ is contained in the image of $\phi^*$. To do so, we consider another commutative diagram where the bottom row is exact.
\begin{equation} \label{md2}
\begin{tikzcd}
& A^{*}((\pp \E^\vee)^n)  \arrow{d}[swap]{\psi^*} \arrow{dr}{\psi'^*} & \\
A^{*-1}(\Sigma_{3,n}) \arrow{r}{\jmath_*} & A^{*}(Y_n) \arrow{r} & A^{*}(\Sigma_{2,n}) \arrow{r} & 0.
\end{tikzcd}
\end{equation}
By Lemma \ref{res}, $\psi'^*$ is surjective. Moreover, by the projection formula and Lemma \ref{res}, the image of $\jmath$ is generated as an $A^*((\p \E^{\vee})^n)$ module by the fundamental class $[\Sigma_{3,n}]\in A^*(Y_n)$. Therefore, any class $\alpha\in A^*(Y_n)$ can be written as
\[
\alpha = \psi^*\alpha_0+[\Sigma_{3,n}]\psi^*\alpha_1=\iota^*\phi^*\alpha_0+[\Sigma_{3,n}]\iota^*\phi^*\alpha_1,
\]
where $\alpha_i\in A^*((\pp \E^\vee)^n)$.
By the projection formula,
\[
\iota_*\alpha=[Y_n]\phi^*\alpha_0+[\Sigma_{3,n}]\phi^*\alpha_1,
\]
where now the fundamental class $[\Sigma_{3,n}]$ is a class on $\H^{\Diamond}_{4,7,n}$.
It thus suffices to show that the classes $[Y_n]$ and $[\Sigma_{3,n}]$ are in the image of $\phi^*$. 

By \cite[Lemma 4.8]{CanningLarson789}, $[\overline{\Sigma_{\ell}}]$ is in the image of $A^*(\B) \to A^*(\H_{4,7}^\Diamond)$. Because $[\Sigma_{3,n}]$ is the pullback of $[\Sigma_3]$ along $A^*(\H_{4,7}^{\Diamond}) \to A^*(\H_{4,7,n}^{\Diamond})$ and $[Y_n]$ is the pullback of $[\overline{\Sigma_{2}}]$, both $[\Sigma_{3,n}]$ and $[Y_n]$ are in the image of $A^*(\B)\to A^*(\H_{4,7,n}^{\Diamond})$. Hence, they are in the image of $\phi^*$. 
\end{proof}

Recall that proper, surjective maps induce surjective maps on rational Chow groups.
Since the map $\beta_n'\colon \H_{4,7,n}^\Diamond \to \M_{7,n} \smallsetminus \M_{7,n}^3$ is proper with image $\M_{7,n}^4 \smallsetminus \M_{7,n}^3$, every class supported on the tetragonal locus is the pushforward of a class from $\H_{4,7,n}^\Diamond$.
Combining Lemmas \ref{tot} and \ref{all} therefore proves the following.

\begin{lem}
If $n \leq 3$, then all classes supported on $\M_{7,n}^4 \smallsetminus \M_{7,n}^3$ are tautological.
\end{lem}

\subsection{The pentagonal locus}
It remains to study the locus $\M_{7,n}^\circ = \M_{7,n} \smallsetminus \M_{7,n}^4$  of curves of gonality exactly $5$.
Mukai showed that every curve in $\M_{7}^{\circ}$ is realized as a linear section of the orthogonal Grassmannian in its spinor embedding $\OG(5,10) \hookrightarrow \pp^{15}$ \cite{Mukai7}. 
To take advantage of this construction, we first develop a few lemmas about the orthogonal Grassmannian.

\subsubsection{The orthogonal Grassmannian}
Let $\V$ be the universal rank $10$ bundle on $\BSO_{10}$. The universal orthogonal Grassmannian is the quotient stack $[\OG(5, 10)/\SO_{10}]$, which we think of as the orthogonal Grassmann bundle with structure map $\pi\colon \OG(5,\V) \to \BSO_{10}$. By construction, the pullback of $\V$ along $\pi$ satisfies  $\pi^* \V = \U \oplus \U^{\vee}$ where $\U$ is the universal rank $5$ subbundle on $\OG(5, \V)$.

\begin{lem} \label{n1}
The stack
\[\OG(5,\V) \cong [\OG(5, 10)/\SO_{10}]\]
has the CKgP. Moreover, its Chow ring is freely generated by the Chern classes of $\U$.
\end{lem}
\begin{proof}
Let $V = \mathrm{span}\{e_1, \ldots, e_{10}\}$ be a fixed $10$-dimensional vector space with quadratic form
\[Q = \left(\begin{matrix} 0 & I_5 \\ I_5 & 0 \end{matrix}\right).\]
Let $U = \mathrm{span}\{e_1, \ldots, e_5\}$, which is an isotropic subspace. The stabilizer of $\SO_{10}$ acting on $\OG(5,10)$ at $U$ is
\[\mathrm{Stab}_U = \left\{M = \left(\begin{matrix} A & B \\ 0 & D \end{matrix}\right) : M^TQM =Q\right\} \subset \SO_{10}.   \]
Expanding, we have
\[M^T Q M = \left(\begin{matrix} A^T & 0 \\ B^T & D^T \end{matrix}\right)\left(\begin{matrix} 0 & I_5 \\ I_5 & 0 \end{matrix}\right)\left(\begin{matrix} A & B \\ 0 & D \end{matrix}\right) = \left( \begin{matrix} 0 & A^TD \\ D^TA & B^TD + D^TB \end{matrix}\right).\]
Thus, $\mathrm{Stab}_U$ is defined by the conditions $D = (A^T)^{-1}$ and $B^TD + D^TB = 0$. 

Note that $\mathrm{Stab}_U$ is a maximal parabolic subgroup and $\OG(5, 10) = \SO_{10}/\mathrm{Stab}_U$.
As such, the quotient $[\OG(5,10)/\SO_{10}]$ is equivalent to the classifying stack $\mathrm{BStab}_U$. To gain a better understanding of the latter, consider
the group homomorphism 
\[\GL_5 \hookrightarrow \mathrm{Stab}_U, \qquad A\mapsto \left(\begin{matrix} A & 0 \\ 0 & (A^T)^{-1}\end{matrix}\right).\] 
For fixed $D$, the condition $B^TD + D^TB = 0$ is linear in $B$. Specifically, it says that, $B$ lies in the $(D^{T})^{-1}$ translation of the $\mathbb{A}^{10}$ of skew symmetric $5 \times 5$ matrices.
In particular, the cosets of the subgroup $\GL_5 \hookrightarrow \mathrm{Stab}_U$ are isomorphic to affine spaces $\mathbb{A}^{10}$. In other words, the induced map on classifying spaces $\BGL_5 \to \mathrm{BStab}_U \cong \OG(5, \V)$ is an affine bundle.
It follows that 
$\OG(5, \V)$ has the CKgP by 
Proposition \ref{CKgPprops}\eqref{affine} and \eqref{BGL}.

Furthermore,
 by construction, the tautological subbundle $\U$ on $\OG(5, \V)$ pulls back to the tautological rank $5$ bundle on $\BGL_5$.
It follows that
$A^*(\mathrm{BStab}_U) \cong A^*(\mathrm{BGL}_5)$, and is freely generated by the Chern classes of the tautological bundle.
\end{proof}

\begin{rem}
There is also a natural map $\mathrm{Stab}_U \to \GL_5$ that sends $\left(\begin{matrix} A & B \\ 0 & D \end{matrix}\right)$ to $A$. The kernel of $\mathrm{Stab}_U \to \GL_5$  is the subgroup $G \cong (\gg_a)^{10}$ where $A = I_5, D = I_5$ and $B + B^T = 0$. This shows that $\mathrm{Stab}_U$ is actually a semi-direct product $G \rtimes \GL_5$. The map $\mathrm{BStab}_U \to \BGL_5$ is a $\mathrm{B}G$-banded gerbe.
\end{rem}

\begin{lem} \label{OGbun}
Let $V$ be a rank $2\nu$ vector bundle with quadratic form on $X$. 
\begin{enumerate}
    \item The rational Chow ring of $\OG(\nu, V)$ is generated over the Chow ring of $X$ by the Chern classes of the tautological subbundle.
    \item If $X$ has the CKgP, then $\OG(\nu, V)$ has the CKgP.
\end{enumerate}
\end{lem}
\begin{proof}
The argument is very similar to that for Grassmannians in type A. 

First consider the case when $X$ is a point. The orthogonal Grassmannian $\OG(\nu, 2\nu)$ is stratified by Schubert cells, each of which is isomorphic to an affine space. By Proposition \ref{CKgPprops}\eqref{stratification}, it follows  that $\OG(\nu, 2\nu)$ has the CKgP.
Moreover, the fundamental classes of these cells are expressed in terms of the Chern classes of the tautological quotient or subbundle via a Giambelli formula \cite[p. 1--2]{KT}. Note that this formula involves dividing by $2$, so it is important that we work with rational coefficients for this claim. In conclusion, the Chern classes of the tautological subbundle generate the Chow ring of $\OG(\nu, 2\nu)$.

More generally for a fiber bundle $f\colon \OG(\nu, V) \to X$, to prove (1), we stratify $X$ into locally closed subsets $X_i$ over which $V$ is trivial, such that $\overline X_i  \supset X_j$ for $i \leq j$ and $X_i$ has codimension at least $i$. To check that the desired classes generate $A^k(\OG(\nu, V))$ for a given $k$, it suffices to show that they generate $A^k(f^{-1}(X_0 \cup X_1 \cup \cdots \cup X_k))$.

Over each piece of the stratification, $f^{-1}(X_i) = X_i \times \OG(\nu, 2\nu)$. Since $\OG(\nu,2\nu)$ has the CKgP, the Chow ring of $f^{-1}(X_i) = X_i \times \OG(\nu, 2\nu)$ is generated by $A^*(X_i)$ and restrictions of the Chern classes from the tautological subbundle on $\OG(\nu, V)$. By excision and the push-pull formula, the Chow ring of any finite union $f^{-1}(X_0) \cup f^{-1}(X_1) \cup \cdots \cup f^{-1}(X_k)$ is generated by the desired classes.

Finally (2) follows from (1) exactly as in \cite[Lemma 3.7]{CL-CKgP}. 
\end{proof}

\begin{cor} \label{ogn}
For any $n \geq 1$, the $n$-fold fiber product
\[\OG(5, \V)^n := \OG(5, \V) \times_{\BSO_{10}} \cdots \times_{\BSO_{10}} \OG(5, \V) \]
has the CKgP and its Chow ring is generated by the Chern classes of the tautological subbundles $\U_1, \ldots, \U_n$ pulled back from each factor.
\end{cor}
\begin{proof}
The case $n = 1$ follows from Lemma \ref{n1}. For $n > 1$, the $n$-fold fiber product is an orthogonal Grassmann bundle over the $(n-1)$-fold fiber product, so the claim follows from Lemma \ref{OGbun}.
\end{proof}

\begin{rem}
By Proposition \ref{CKgPprops}\eqref{proper}, the fact that $\OG(5, \V)$ has the CKgP implies that $\BSO_{10}$ also has the CKgP. 
It should be possible to show that $\BSO_{10}$ has the CKgP (with integral coefficients as well) using the calculation of its Chow ring by Field \cite{Field}. If the calculation there holds over any field, it would show that $\BSO_{10}$ has Totaro's ``weak Chow--K\"unneth Property," which is equivalent to the CKgP by the proof of \cite[Theorem 4.1]{TotaroCKgP}.
\end{rem}

\subsubsection{Review of the Mukai construction}
We first review Mukai's construction and then explain how to modify it for pointed curves.
The canonical model of a pentagonal genus $7$ curve $C \subset \pp^6$ lies on a $10$-dimensional space of  quadrics.
The vector space $W$ of these quadrics is defined by the exact sequence 
\begin{equation} \label{W} 0 \rightarrow W \rightarrow \Sym^2 H^0(C, \omega_C) \to H^0(C, \omega_C^{\otimes 2}) \rightarrow 0.
\end{equation}
For each $p \in C$, the subspace $W_p \subset W$ of quadrics that are singular at $p$ is $5$-dimensional. It appears in the exact sequence (see \cite[Section 3]{Mukai7})
\begin{equation} \label{Wp} 0 \rightarrow W_p \rightarrow \Sym^2 H^0(C, \omega_C(-p)) \to H^0(C,\omega_C(-p)^{\otimes 2}) \rightarrow 0.
\end{equation}
Mukai shows that $W^\vee$ has a canonical quadratic form and $W_p^\perp$ is an isotropic subspace, so one obtains a map $C \to \OG(5, W^\vee)$ via $p \mapsto [W_p^\perp] \in \OG(5, W^\vee)$ \cite[Theorem 0.4]{Mukai7}.

Let $\OG(5, 10) \hookrightarrow \pp S^+$ be the spinor embedding, as in \cite[Section 1]{Mukai7}.
The composition 
\[C \to \OG(5, W^\vee) \cong \OG(5, 10) \hookrightarrow \pp S^+\]
realizes $C$ as a linear section $C = \OG(5,10) \cap \pp^6 \subset \pp S^+$, and $C \subset \pp^6$ is canonically embedded \cite[Theorem 0.4]{Mukai7}. Conversely, if a linear section $\pp^6 \cap \OG(5, 10)$ is a smooth curve, then it is a canonically embedded pentagonal genus $7$ curve \cite[Proposition 2.2]{Mukai7}.
This construction works in families, and Mukai proves that there is an equivalence of stacks
\begin{equation} \label{m1} 
\M_{7}^\circ \cong [(\Gr(7, S^+) \smallsetminus \Delta)/\SO_{10}],
\end{equation}
where $\Delta \subset \Gr(7, S^+)$ is the closed locus of linear subspaces $\pp^6 \subset \pp S^+$ whose intersection with $\OG(5, 10)$ is not a smooth curve \cite[Section 5]{Mukai7}.

The
spinor representation of $\SO_{10}$ corresponds to a rank $16$ vector bundle $\S^+$ on the classifying stack $\BSO_{10}$.
Then, the equivalence \eqref{m1} shows that $\M_7^\circ$ is an open substack of the Grassmann bundle $\Gr(7, \S^+)$ over $\BSO_{10}$:
\begin{equation} \label{gbun}
\begin{tikzcd} \M_7^\circ \ar[hook,r,"\alpha_0"] & \Gr(7, \S^+) \arrow{d} \\ & \BSO_{10}.
\end{tikzcd}
\end{equation}
We note that the tautological subbundle $\E$ on $\Gr(7, \S^+)$ restricts to the dual of the Hodge bundle on $\M_7^\circ$. In other words, if $f\colon \C \to \M_7^{\circ}$ is the universal curve, then $\alpha_0^*\E = (f_*\omega_f)^\vee$.
The universal version of the Mukai construction furnishes a map $\C \to \OG(5, \V)$ over $\BSO_{10}$.

\subsubsection{The Mukai construction with markings}
For $n \leq 4$, we describe $\M_{7,n}^{\circ}$ in a similar fashion to \eqref{gbun}, but this time as an open substack of a Grassmann bundle over $\OG(5,\V)^n$. 
To do so, let $\iota\colon \OG(5, \V) \hookrightarrow \pp \S^+$ be the universal spinor embedding and let $\L := \iota^*\O_{\pp \S^+}(-1)$. The embedding $\iota$ is determined by an inclusion of vector bundles $\L \hookrightarrow \pi^*\S^+$ on $\OG(5, \V)$.
On the $n$-fold fiber product $\OG(5, \V)^n$, let $\L_i$ and $\U_i$ denote the pullbacks of $\L$ and $\U$, respectively, from the $i$th factor. Let $\pi_n \colon \OG(5, \V)^n \to \BSO_{10}$ be the structure map.
 Consider the sum of the inclusions
 \begin{equation} \label{phin}
\bigoplus_{i=1}^n \L_i \xrightarrow{\phi_n} \pi_n^* \S^+.
 \end{equation}
Let $Z_n \subset \OG(5,\V)^n$ be the open substack where $\phi_n$ has rank $n$. In other words, $Z_n$ is the locus where the $n$ points on $\OG(5, \V)$ have independent image under the spinor embedding.

Let $\Q_n$ be the cokernel of $\phi_n|_{Z_n}$, which is a rank $16 - n$ vector bundle on $Z_n$. The fiber of $\Gr(7 - n, \Q_n)$ over $(p_1, \ldots, p_n) \in \OG(5, \V)^n$ parametrizes linear spaces $\pp^6 \subset \pp S^+$ that contain the $n$ points $p_i$.
Thus, we can identify $\Gr(7-n,\Q_n)$ with the locally closed substack
\begin{equation} \label{lc} \left\{(p_1, \ldots, p_n, \Lambda)  : p_i \in \pp \Lambda \text{ and $p_i$ independent} \right\} \subset  
\OG(5, \V)^n \times_{\BSO_{10}} \Gr(7, \S^+).
\end{equation}

\begin{lem} \label{7nckgp}
For $n \leq 4$, there is an open embedding $\alpha_n$ of $\M_{7,n}^\circ$ in the Grassmann bundle
\begin{center}
\begin{tikzcd}
\M_{7,n}^{\circ} \ar[r,hook, "\alpha_n"] & \Gr(7-n, \Q_n) \arrow{d} \\
& Z_n \ar[r,hook] &\OG(5, \V)^n.
\end{tikzcd}
\end{center}
Hence, $\M_{7,n}^{\circ}$ has the CKgP.
\end{lem}
\begin{proof}
Let $f\colon \C \to \M_{7,n}^{\circ}$ be the universal curve and let $\sigma_i\colon \M_{7,n}^{\circ} \to \C$ be the $i$th section.
The universal version of the Mukai construction gives a morphism $\C \to \OG(5, \V)$. Precomposing with each of the sections $\sigma_i$ defines a map $\M_{7, n}^{\circ} \to \OG(5,\V)^n$ over $\BSO_{10}$. We claim that, for $n \leq 4$, the images of $\sigma_1, \ldots, \sigma_n$ must be independent. Indeed, suppose the images of the sections in a fiber, $p_1, \ldots, p_n \in C$, are dependent under the canonical embedding. Then $p_1 + \cdots + p_n$ would give a $g^1_n$ on $C$, but curves in $\M_{7,n}^{\circ}$ have no $g^1_n$ for $n \leq 4$ by definition.

We also have the map $\M_{7,n}^{\circ} \to \M_7^\circ \to \Gr(7, \S^+)$ that sends a curve to its span under the spinor embedding.
Taking the product of these maps over $\BSO_{10}$ yields a map \[\M_{7,n}^\circ \to \OG(5, \V)^n \times_{\BSO_{10}} \Gr(7, \S^+).\]
This map sends a family of pointed curves $(C, p_1, \ldots, p_n)$ over a scheme $T$ to the data of sections $p_i\colon T \to C \to \OG(5, W^\vee)$ and the subbundle of the spinor representation of $W^\vee$ determined by the span of the fibers of $C \subset \OG(5, W^\vee) \subset \pp S^+$ over $T$. The map evidently factors through the locally closed locus in \eqref{lc}, which we identified with $\Gr(7 - n, \Q_n)$. In fact, the image is precisely $\Gr(7 - n, \Q_n) \smallsetminus \Delta$, where $\Delta$ is the closed locus such that $\pp \Lambda \cap \OG(5, W^\vee)$ is not a family of smooth curves. Indeed, on the complement of $\Delta$, an inverse map $\Gr(7 - n, \Q_n) \smallsetminus \Delta \to \M_{7,n}^{\circ}$ is defined
by sending $(p_1, \ldots, p_n, \Lambda)$ to the curve $C = \pp \Lambda \cap \OG(5, W^\vee)$ together with the sections $p_i \in \pp \Lambda \cap \OG(5, W^\vee) = C$. 
 
By Corollary \ref{ogn}, we know $\OG(5, \V)^n$ has the CKgP. To complete the proof, apply Proposition \ref{CKgPprops}\eqref{open} and \eqref{Grassmann}.
\end{proof}

We now identify the restrictions of the universal bundles on $\OG(5, \V)^n$ and $\Gr(7 - n, \Q_n)$ to $\M_{7,n}^{\circ}$ along $\alpha_n$. 
\begin{lem} \label{Ui}
Let $f\colon \C \to \M_{7,n}^{\circ}$ be the universal curve and let $\sigma_i$ denote the image of the $i$th section.
The vector bundle $\alpha_n^*\, \U_i$ sits in an exact sequence
\begin{equation} \label{wprel}0 \rightarrow \alpha_n^* \, \U_i \rightarrow \Sym^2 f_*(\omega_f(-\sigma_i)) \rightarrow f_*((\omega_f(-\sigma_i))^{\otimes 2}) \rightarrow 0.
\end{equation}
In particular, all classes pulled back from $\OG(5, \V)^n$ to $\M_{7,n}^\circ$ are tautological.
\end{lem}
\begin{proof}
The composition of $\alpha_n$ with projection onto the $i$th factor of $Z_n$ is the map that sends a pointed curve $(C, p_1, \ldots, p_n)$ to the image of $p_i$ under the canonical map from $C$ to $\OG(5, W^\vee) \cong \OG(5, 10)$. By construction, the fiber at $p_i$ of the universal rank $5$ bundle on $\OG(5,10)$ is $W_p^\perp$. Equation \eqref{wprel} is the relative version of \eqref{Wp}. By Grothendieck--Riemann--Roch, the middle and right terms in \eqref{wprel} have tautological Chern classes.
It follows that the Chern classes of $\alpha_n^*\U_i$ are also tautological. The last claim now follows from Corollary \ref{ogn}.
\end{proof}

\begin{lem} \label{Lpsi}
We have $c_1(\alpha_n^*\L_i) = -\psi_i$ in $A^1(\M_{7,n}^{\circ})$.
\end{lem}
\begin{proof}
Let $f\colon\C \to \M_{7,n}^{\circ}$ be the universal curve. The line bundle $\alpha^*\L_i$ is the pullback of $\O_{\pp \S^+}(-1)$ along the composition
\[\M_{7,n}^{\circ} \xrightarrow{\sigma_i} \C \rightarrow \OG(5, \V) \xrightarrow{\iota} \pp \S^+.\]
But the above composition also factors as
\[\M_{7,n}^{\circ} \xrightarrow{\sigma_i} \C \xrightarrow{|\omega_f|} \pp (f_* \omega_f)^\vee \to \pp \S^+,\]
and $\O_{\pp \S^+}(-1)$ restricts to $\O_{\pp (f_* \omega_f)^\vee}(-1)$. Hence, $\alpha_n^*\L_i$ is $\sigma_i^* \O_{\pp (f_* \omega_f)^\vee}(-1) = (\sigma_i^*\omega_f)^\vee$.
\end{proof}

\begin{lem} \label{En}
Let $\E_n$ be the tautological rank $7 - n$ subbundle on $\Gr(7 - n, \Q_n)$.
The pullback $\alpha_n^*\E_n$ sits in an exact sequence
\begin{equation} \label{lfe} 0 \rightarrow \bigoplus_{i=1}^n \alpha_n^*\L_i \rightarrow (f_*\omega_f)^\vee \rightarrow \alpha_n^*\E_n \rightarrow 0.
\end{equation}
In particular, the Chern classes of $\alpha_n^* \E_n$ are tautological.
\end{lem}
\begin{proof}
Recall that we defined $\Q_n$ as the cokernel of \eqref{phin}. 
The map onto the second factor in \eqref{lc}, $p\colon \Gr(7 - n, \Q_n) \to \Gr(7, \S^+)$, sends a $7 - n$ dimensional subspace to its preimage under the quotient map $\S^+ \to \Q_n$. Hence, there is an exact sequence on $\Gr(7 - n, \Q_n)$
\[0 \rightarrow \bigoplus_{i=1}^n \L_i \to p^* \E \to \E_n \rightarrow 0.\]
Then note that $\alpha_n^*p^*\E$ is the same as the pullback of $\E$ along $\M_{7,n}^\circ \to \M_{7}^\circ \to \Gr(7, \S^+)$, which is the dual of the Hodge bundle. The last claim follows by combining the exact sequence \eqref{lfe} with Lemma \ref{Lpsi}.
\end{proof}

\begin{lem}
The Chow ring of $\M_{7,n}^{\circ}$ is generated by tautological classes.
\end{lem}
\begin{proof}
By Lemma \ref{7nckgp}, we know that $\M_{7,n}^\circ$ is an open substack of $\Gr(7 - n, \Q_n)$. The Chow ring of the latter is generated by pullbacks of classes from $\OG(5, \V)^n$ and by the Chern classes of the tautological subbundle $\E_n$. By Lemmas \ref{Ui} and \ref{En} respectively, both of these collections of classes restrict to tautological classes on $\M_{7,n}^\circ$.
\end{proof}

\section{Applications to even cohomology}

Here we use the results from Sections \ref{toprestrictions}, \ref{agp}, and \ref{sec:7CKgP} to prove Theorem \ref{thm:taut}.

\subsection{The degree 4 cohomology of \texorpdfstring{$\Mb_{g,n}$}{Mbargn} is tautological}

\begin{proof}[Proof of Theorem \ref{thm:taut}(1)] 
By \cite{ArbarelloCornalba}, we know that $H^k(\Mb_{g,n})$ is tautological for $k \leq 3$.
Therefore, using Theorem \ref{thm:fgSTE}, we deduce that $H^4(\Mb_{g,n})$ is contained the STE generated by $H^4(\Mb_{g',n'})$ for $g' < 7$ and $n' \leq 4$.
By Proposition  \ref{openCKgPresults}, $W_4H^4(\M_{g',n'})$ is tautological for $g'$ and $n'$
 in this range, and it follows that $H^4(\Mb_{g,n})$ is tautological. 
\end{proof}

\begin{proof}[Proof of Theorem \ref{thm:taut}(2)]
    By Theorem \ref{thm:taut}(1) and \cite{ArbarelloCornalba}, all classes in $H^6(\Mb_{g,n})$ that are pushed forward from the boundary are tautological. Also, $H^6(\M_{g,n})$ is stable and hence tautological for $g \geq 10$ \cite{Wahl}. It follows that $H^6(\Mb_{g,n})$ is tautological for $g \geq 10$.
\end{proof}

\begin{rem}
    By Lemma \ref{blue}, to show that $H^6(\Mb_{g,n})$ is tautological for all $g$ and $n$, it would suffice to show this for $g\leq 9$ and $n\leq 6$. In principle, this can be checked computationally as follows. By Theorem \ref{thm:taut}\eqref{it:homology}, we know that $H_6(\Mb_{g,n})$ is tautological for all $g$ and $n$. Thus, if the intersection pairing
    \[
RH^6(\Mb_{g,n})\times RH_6(\Mb_{g,n})\rightarrow \qq
    \]
    is perfect for $g\leq 9$ and $n\leq 6$, then $H^6(\Mb_{g,n})$ is tautological. In principle, one could compute this pairing using the Sage package admcycles \cite{admcycles}. In practice, however, this computation is too memory intensive to carry out.
\end{rem}

\subsection{The low degree even homology of \texorpdfstring{$\Mb_{g,n}$}{Mbargn} is tautological}

Let $H_k^{BM}$ denote Borel--Moore homology with coefficients in $\qq$ or $\qq_{\ell}$, together with its mixed Hodge structure or Galois action.
From the long exact sequence in Borel--Moore homology, we have a right exact sequence
\begin{equation}\label{BMsequence}
    H_k(\tilde{\partial \M_{g,n}})\rightarrow H_k(\Mb_{g,n})\rightarrow W_{-k} H_k^{BM}(\M_{g,n})\rightarrow 0
\end{equation}
for all $k$.

\begin{proof}[Proof of Theorem \ref{thm:taut}(3)] We now show that $H_k(\Mb_{g,n})$ is tautological for even $k \geq 14$. 
For $g=0$, we have $H^*(\Mb_{0,n}) = RH^*(\Mb_{0,n})$\cite{Keel}. For $g \geq 1$, by \cite[Proposition~2.1]{BergstromFaberPayne} and the duality between $H_k^{BM}$ and $H^k_c$, we have:
    \begin{equation*}
    H_k^{BM}(\M_{g,n})=0 \text{ for } \begin{cases}
     k<2g \text{ and } n=0,1; \\
     k<2g-2+n \text{ and } n\geq 2.
    \end{cases}
\end{equation*}
Combining this with the exact sequence \eqref{BMsequence}, we reduce inductively to the finitely many cases $k\geq 2g$ and $n=0,1$ or $k\geq 2g-2+n$ and $n\geq 2$. When $g\geq 3$, in all of these cases with $k\leq 14$ we know that $\M_{g,n}$ has the CKgP and $A^*(\M_{g,n})=R^*(\M_{g,n})$ (see Table~\ref{ckgptable}). By Proposition \ref{openCKgPresults}, it follows that $W_{-k}H^{BM}_k(\M_{g,n})$ is tautological, so again by induction and the exact sequence \eqref{BMsequence} we reduce to the cases of $g=1,2$. When $g=1$, all even cohomology is tautological \cite{Petersengenus1}. When $g=2$, we apply Corollary \ref{g2lowdegreetaut}.
\end{proof}
\begin{cor}\label{homologymotives14}
    As Hodge structures or Galois representations, we have
    \[
    H^{2d_{g,n}-14}(\Mb_{g,n})^{\semis}\cong \bigoplus \mathsf{L}^{d_{g,n}-7}
    \]
\end{cor}
\begin{proof}
  We proved $H_{14}(\Mb_{g,n})$ is tautological and hence algebraic. Applying Poincar\'e duality yields the corollary.
\end{proof}

\section{The thirteenth and fifteenth homology of \texorpdfstring{$\Mb_{g,n}$}{Mbargn}} \label{h13sec}
In this section, we prove Theorems \ref{thm:13} and \ref{thm:15}. As a corollary, we obtain restrictions on the Hodge structures and Galois representations appearing in $H_{13}(\Mb_{g,n})$ and $H_{15}(\Mb_{g,n})$. 
\subsection{Thirteenth homology}
As usual, we argue by induction on $g$ and $n$.  First, we must treat one extra base case.

\begin{lem}\label{H13Mb211}
   The group $H^{13}(\Mb_{2,11})$ is in the STE generated by $H^{11}(\Mb_{1,11})$.
\end{lem}
\begin{proof}
We showed in \cite{CanningLarsonPayne} that $H^{11}(\Mb_{g,n})$ is in the STE generated by $H^{11}(\Mb_{1,11})$ for all $g$ and $n$.
Next, we note that $\M_{2,10}$ has the CKgP (see Table \ref{ckgptable}), so $H^{13}(\Mb_{2,10})$ consists of classes pushed forward from the boundary by Proposition \ref{openCKgPresults}. In particular, $H^{13}(\Mb_{2,10})$ also lies in this STE.
The claim then follows from Lemma \ref{odd}.
\end{proof}

\begin{proof}[Proof of Theorem \ref{thm:13}] 
The proof is similar to that of Theorem \ref{thm:taut}(3). 
When $g=0$, all odd cohomology vanishes \cite{Keel}, so we can assume $g\geq 1$. By \cite[Proposition 2.1]{BergstromFaberPayne} and the duality between $H_{13}^{BM}$ and $H^{13}_c$, we have that for $g\geq 1$
    \begin{equation*}
    H^{BM}_{13}(\M_{g,n})=0 \text{ for } \begin{cases}
     13<2g \text{ and } n=0,1 \\
     13<2g-2+n \text{ and } n\geq 2.
    \end{cases}
\end{equation*}
Combining this with the exact sequence \eqref{BMsequence}, we reduce inductively to the finitely many cases where $13\geq 2g$ and $n=0,1$ or $13\geq2g-2+n$ and $n\geq 2$. All such cases except for $(g,n)$ in $Z = \{ (2,11), (1,11),(1,12),(1,13)\}$ have the CKgP (see Table \ref{ckgptable}). By Proposition \ref{openCKgPresults}, it follows that $W_{-13}H_{13}^{BM}(\M_{g,n})=0$ for $(g,n) \not \in Z$. By induction and the exact sequence \eqref{BMsequence}, we reduce to the exceptional cases of $(g,n) \in Z$. The case $(g,n)=(2,11)$ follows from Lemma \ref{H13Mb211} and Hard Lefschetz. The cases where $g=1$ follow from Proposition~\ref{genus1weightk} and induction on $n$. 
\end{proof}
Following the proof strategy above, we obtain the following corollary.
\begin{cor}\label{homologymotives13}
    As Hodge structures or Galois representations, we have
    \[
    H^{2d_{g,n}-13}(\Mb_{g,n})^{\semis}\cong \bigoplus\mathsf{L}^{d_{g,n}-12}\mathsf{S}_{12}
    \]
    for all $g$ and $n$.
\end{cor}
\begin{proof}

From the proof above, we see that $W_{-13}H_{13}^{BM}(\M_{g,n})=0$ for $(g,n)\notin Z$, and so
    \[
    H_{13}(\tilde{\partial \M_{g,n}})\rightarrow H_{13}(\Mb_{g,n})
    \]
is surjective. Noting that pushforward along a gluing map just induces a Tate twist, by induction on $g$ and $n$, we reduce to the cases $(g, n) \in Z$. The case $(g,n)=(2,11)$ follows from Lemma \ref{H13Mb211} and Hard Lefschetz. The cases where $g=1$ follow from Proposition~\ref{genus1weightk} and induction on $n$. 
\end{proof}

\subsection{Fifteenth homology}
Again, we will argue by induction on $g$ and $n$. We first need to treat two additional base cases in genus $2$. The difficult one is with $12$ markings.
\begin{lem} \label{mainl}
$H^{15}(\Mb_{2,12})$ is generated by classes pushed forward from the boundary. Hence, $H^{15}(\Mb_{2,12})$ is contained in the STE generated by $H^{11}(\Mb_{1,11})$.
\end{lem}
The proof of Lemma \ref{mainl} involves calculating intersection pairings via decorated graphs. In order to simplify pictures for and language surrounding decorated graphs, we use the following conventions:
\begin{enumerate}
    \item A white circle $\circ$ is a vertex of genus $1$ 
    \item A black circle $\bullet$ is a vertex of genus $0$
    \item All markings that are not drawn elsewhere are understood to be on the leftmost genus $1$ vertex
    \item The decoration $\omega_i$ on a genus $1$ vertex is shorthand for $f_i^*\omega$ where $f_i\colon  \Mb_{1,12} \to \Mb_{1,11}$ forgets the $i$th marking and $\omega \in H^{11}(\Mb_{1,11})$. Similarly, $\omega_{ij}$ is shorthand for $f_{ij}^*\omega$ where $f_{ij} \colon \Mb_{1,13} \to \Mb_{1,11}$ forgets the $i$th and $j$th markings.
    \item As we vary the decoration $\omega$ over $H^{11}(\Mb_{1,11}) \cong \mathsf{S}_{12}$, pushforward along a map that glues $e$ pairs of marked points determines a morphism $\mathsf{L}^e \mathsf{S}_{12} \to H^{11+2e}(\Mb_{g,n})$. We say that a collection of $\omega$-decorated graphs ``is a basis for $H^k(\Mb_{g,n})$" if the corresponding morphisms $\mathsf{L}^e \mathsf{S}_{12} \to H^{k}(\Mb_{g,n})$ form a basis for $\Hom(\mathsf{L}^e\mathsf{S}_{12}, H^{k}(\Mb_{g,n}))$ and we have $H^k(\Mb_{g,n}) \cong \Hom(\mathsf{L}^e\mathsf{S}_{12}, H^{k}(\Mb_{g,n})) \otimes \mathsf{L}^e \mathsf{S}_{12}$.
\end{enumerate}

Bergstr\"om and Faber's implementation \cite{BergstromData} of the Getzler--Kapranov formula in genus $2$  shows that, viewed as Galois representations,
\begin{equation} \label{thedim} 
H^{15}(\Mb_{2,12})^{\semis} = 836 \, \mathsf{L}^2\mathsf{S}_{12}
\end{equation}
In fact, we will produce $836$ independent maps $\mathsf{L}^2\mathsf{S}_{12} \to H^{15}(\Mb_{2,12})$ that arise through pushforward from the boundary. From this, we will conclude that $H^{15}(\Mb_{2,12})$ is semi-simple and decomposes as a sum of these $836$ copies of $\mathsf{L}^2\mathsf{S}_{12}$.

As discussed above, an $\omega$-decorated graph with $e$ edges represents a morphism $\mathsf{L}^e \mathsf{S}_{12} \to H^{11 + 2e}(\Mb_{g,n})$.
In Figure \ref{pullbackchart} we list $891$ such decorated graphs for $H^{15}(\Mb_{2,12})$ and record particular K\"unneth components of their pullbacks to boundary divisors. We then argue that there are $55$ relations among these graphs, which allows us to conclude that there are $836$ independent graphs among them.

Before proving Lemma \ref{mainl}, we require some explicit bases for $H^{13}(\Mb_{g,n})$ for small $g, n$. 
Recall that Getzler \cite{Getzler} showed that $H^{13}(\Mb_{1,n})^{\semis} = \bigoplus \lstw$
for all $n$. 
\begin{lem} \label{112}
The following $11$ $\omega$-decorated graphs form a basis for $H^{13}(\Mb_{1,12})$.
\[\includegraphics[width=.8in]{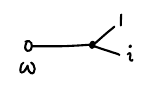}\]
\end{lem}
\begin{proof}
Consider the block diagonal pairing against the basis for $H^{11}(\Mb_{1,12})$ given by the set of genus $1$ vertices decorated by $\omega_2, \ldots, \omega_{12}$.
\end{proof}

\begin{lem} \label{h13m113}
The following $429$ $\omega$-decorated graphs form a basis for $H^{13}(\Mb_{1,13})$.
\[\includegraphics[width=3.5in]{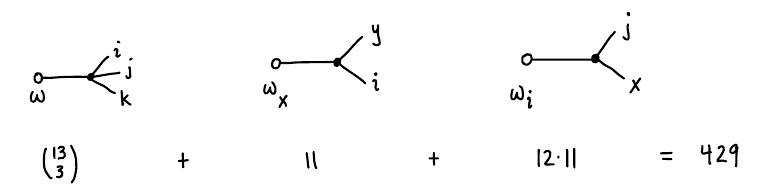}\]
Above, $x, y$ are any two markings. In the last group of classes, $i$ or $j$ can be equal to $y$. 
\end{lem}
\begin{proof}
Getzler \cite{Getzler} proved that $H^{13}(\Mb_{1,13})^{\semis} = 429 \, \lstw$, so it suffices to show that the images of these maps span.
We know that $H^{13}(\Mb_{1,13})$ is pushed forward from the boundary by Proposition \ref{genus1weightk}.
Modulo classes of the first kind, the second two kinds are pulled back from $H^{13}(\Mb_{1,12})$. Pulling back the relations from $H^{13}(\Mb_{1,12})$, we can ensure that a fixed marking such as $y$ in the first case or $x$ in the next case is on the genus $0$ vertex.
\end{proof}

\begin{lem} \label{14pts}
Let $\{x, y, 1, \ldots, 12\}$ be a set of $14$ markings.
The following $6006$ $\omega$-decorated graphs form a basis for $H^{13}(\Mb_{1,14})$.
\[\includegraphics[width=5.5in]{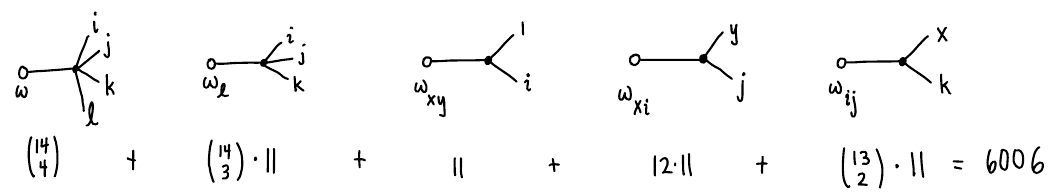}\]
In the last grouping above, $i, j,$ or $k$ can equal $y$.
\end{lem}
\begin{proof}
Getzler \cite{Getzler} proved that $H^{13}(\Mb_{1,14})^{\semis}=6006\,\lstw$, so it suffices to show that the images of these maps span. Modulo classes of the first and second kinds above, the last three kinds are pulled back from $\Mb_{1,12}$. Thus, using the relations on $H^{13}(\Mb_{1,12})$ we can ensure that a fixed marking among those $12$ is on the genus $0$ vertex.
\end{proof}

\begin{lem} \label{12pts}
The following  $264$ $\omega$-decorated graphs form a basis for $H^{13}(\Mb_{2,12})$.
\[\includegraphics[width=3.5in]{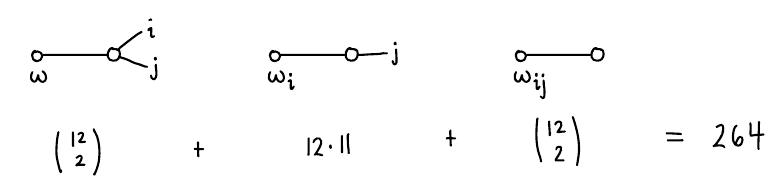}\]
\end{lem}
\begin{proof}
Bergstr\"om and Faber's implementation \cite{BergstromData} of the Getzler--Kapranov formula in genus $2$  shows that, as Galois representations,
    $H^{13}(\Mb_{2,12})^\semis = 264 \, \lstw$.
Thus, it suffices to show that the above $264$ maps are independent. We verify this by computing their pullbacks to $H^{13}(\Mb_{1,14})$ and using Lemma \ref{14pts} to see that those maps are independent. Indeed, the pullback is given by the graph that replaces the rightmost genus $1$ vertex with a genus $0$ vertex and adds two markings labeled $x, y$ to that vertex.
\end{proof}

\begin{rem}
Theorem \ref{thm:13} says that $H_{13}(\Mb_{g,n})$ is contained in the STE generated by $H^{11}(\Mb_{1,11})$. Thus, it has a graphical presentation, in which the generators are graphs of the sort appearing in Lemmas \ref{112}--\ref{12pts}. In forthcoming work, we will show that $H^{13}(\Mb_{g,n})$ is also contained in this STE. There, we take a more systematic approach and provide a complete list of relations among the corresponding decorated graph generators.
\end{rem}

\begin{proof}[Proof of Lemma \ref{mainl}]
Figure \ref{pullbackchart} represents a matrix with $891$ rows.
We first show that the matrix has rank $825$ and describe the $66$ relations among the rows.

First notice that the rows in A1 are the only rows with nonzero entries in the first column. Next, rows in A2, A3, and A4 have entries in column 2 that are independent from the other entries that appear in that column for later rows (by Lemma \ref{h13m113}).
This shows that the rows in A1 through A4 are independent and independent from the remaining five rows. Meanwhile, the bottom four rows are block upper triangular. This shows that the rank of the matrix is at least $825$.

\begin{figure} 
\centering
\includegraphics[width=6.5in]{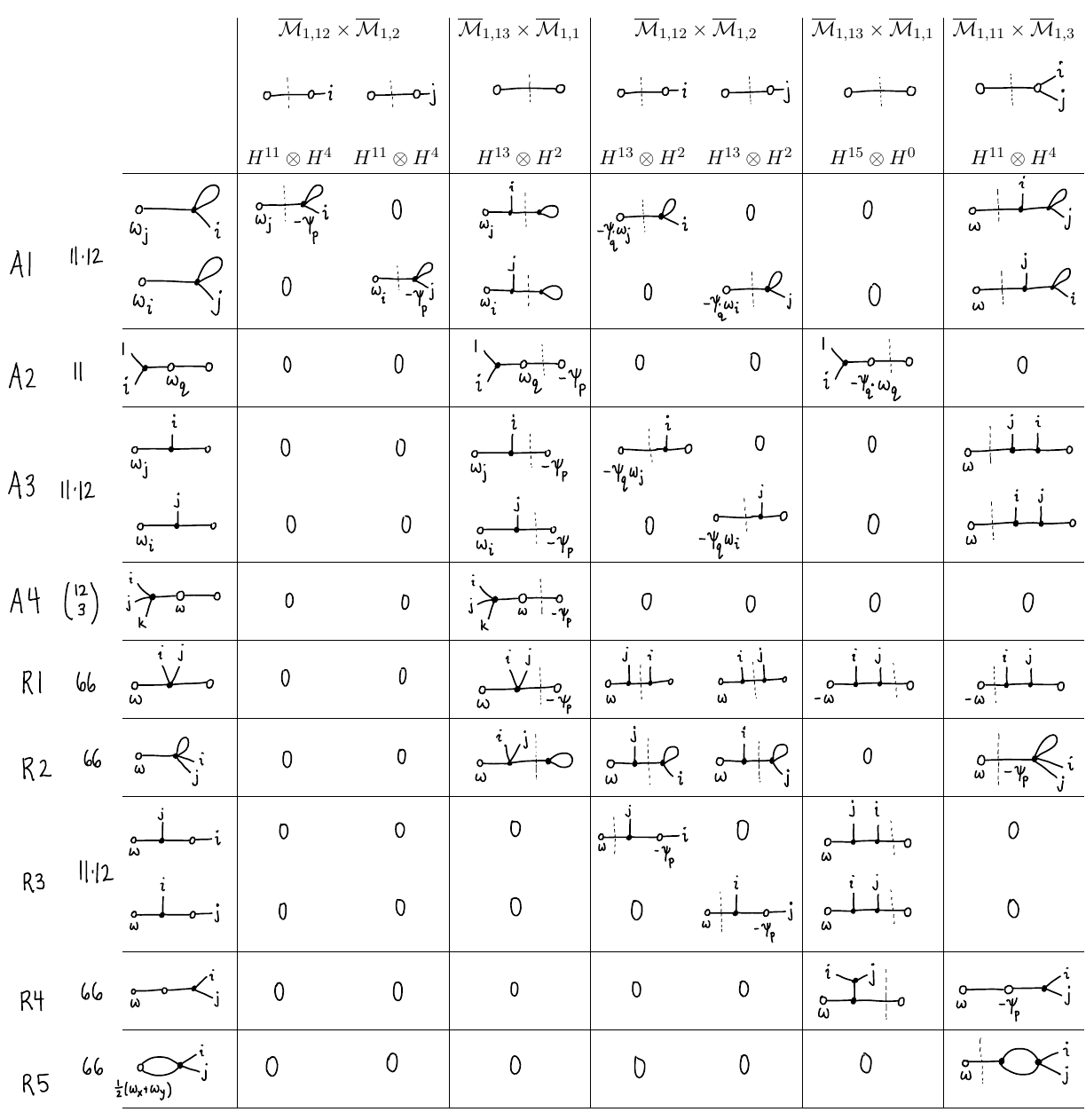}
\caption{Selected K\"unneth components of pullbacks of the $891$ classes to boundary divisors. In the chart, each row represents several rows of a block matrix. The number of rows in each group is listed to its left. The two half-edges glued to make a boundary divisor (pictured across the top) are labeled $q$ on the left vertex and $p$ on the right vertex.}
\label{pullbackchart}
\end{figure}

Meanwhile, it is not hard to check that the following $66$ relations exist among the rows:
\[R_1(ij) + \frac{1}{24}R_2(ij)  + R_3[i,j] + R_3[j,i] - R_4(ij) + \frac{1}{12} R_5(ij) = 0.\]
To verify these relations one needs to make use of known relations among tautological classes in $H^4(\Mb_{1,3})$ and $H^2(\Mb_{1,2})$ and $H^2(\Mb_{1,1})$.

Now, let
\[\includegraphics[width=5in]{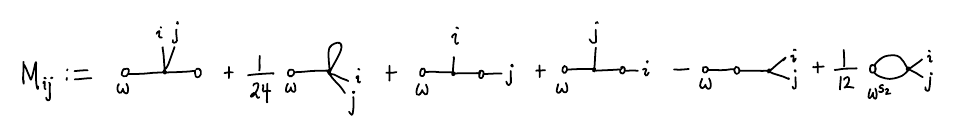}\]
We claim $M_{ij} \neq 0$. To see this, consider the
intersection of $M_{ij}$ with $R_5(ik)$ for $k \neq j$. The intersection of $R_5(ik)$ with the first $5$ terms in the above equation vanish since there is no $H^{11}$ on a genus $1$ vertex with less than $11$ markings. Meanwhile, the intersection $R_5(ij) \cdot R_5(ik)$ is nonzero using the usual rules for intersecting decorated graphs as in \cite{GraberPandharipande}. To see this, one should take the $\omega$ decoration on the $\Mb_{1,11}$ vertex in $R_5(ij)$ to be dual to the decoration used on the $\Mb_{1,11}$ vertex in $R_5(ik)$. 

Finally, the $\mathbb{S}_{12}$ action on the symbols $M_{ij}$ is $\mathrm{Ind}_{\ss_{10}}^{\ss_{12}}(V_{1^{10}}) = V_{2,1^{10}} \oplus V_{3,1^9}$, which is a sum of irreducible representations of dimension $11$ and $55$. It follows that there are either $11$ or $55$ relations among the $M_{ij}$. There cannot be only $11$ relations among the $M_{ij}$ because then there would be more than $836$ independent copies of $\mathsf{L}^2\mathsf{S}_{12}$ in $H^{15}(\Mb_{2,12})$, violating the dimension count \eqref{thedim}. Hence, there must be exactly $55$ relations among the $M_{ij}$. This shows that the $891$ copies of $\mathsf{L}^2\mathsf{S}_{12}$ listed on the left of the table generate an $836$-dimensional space, so they span all of $H^{15}(\Mb_{2,12})$.
\end{proof}

\begin{rem}
One could instead prove that the $891$ classes span a space of rank $836$ by computing the pairing of this space against itself. By the argument above, the pairing must have rank $836$. 
\end{rem}

Finally, applying Lemma \ref{odd}(a) together with Lemmas \ref{mainl} and \ref{12pts}, and Corollary \ref{1ns}, we obtain one more base case.

\begin{lem}\label{13markedpoints}
$H^{15}(\Mb_{2,13})$ lies in the STE generated by $H^{11}(\Mb_{1,11})$.
\end{lem}

\begin{proof}[Proof of Theorem \ref{thm:15}]
The proof is similar to that of Theorem \ref{thm:13}. When $g=0$, all odd cohomology vanishes \cite{Keel}, so we can assume $g\geq 1$. By \cite[Proposition 2.1]{BergstromFaberPayne} and the duality between $H_{15}^{BM}$ and $H^{15}_c$, we have that for $g\geq 1$
    \begin{equation*}
    H^{BM}_{15}(\M_{g,n})=0 \text{ for } \begin{cases}
     15<2g \text{ and } n=0,1 \\
     15<2g-2+n \text{ and } n\geq 2.
    \end{cases}
\end{equation*}
Combining this with the exact sequence \eqref{BMsequence}, we reduce inductively to the finitely many cases $15\geq 2g$ and $n=0,1$ or $15\geq 2g-2+n$ and $n\geq 2$. In such cases when $g\geq 3$, $\M_{g,n}$ has the CKgP, and so by induction and the exact sequence \eqref{BMsequence}, we reduce to the cases where $g=1$ and $11\leq n \leq 14$, or $g=2$ and $11\leq n\leq 13$. When $g=2$ and $n=11$, the result follows from Hard Lefschetz and Lemma \ref{H13Mb211}. When $g=2$ and $12\leq n\leq 13$, the result follows from Lemmas \ref{mainl} and \ref{13markedpoints}. For $g = 1$, the claim follows from Proposition \ref{genus1weightk}.
\end{proof}

\begin{cor}\label{homologymotives15}
    As Hodge structures or Galois representations, we have
    \[
    H^{2d_{g,n}-15}(\Mb_{g,n})^{\semis}\cong \bigoplus\mathsf{L}^{d_{g,n}-13}\mathsf{S}_{12}\oplus \bigoplus \mathsf{L}^{d_{g,n}-15}\mathsf{S}_{16}
    \]
    for all $g$ and $n$.
\end{cor}
\begin{proof}
The proof is similar to that of Corollary \ref{homologymotives13}, using the additional input that $H^{15}(\Mb_{1,15})^
\semis \cong 
\mathsf{S}_{16} \oplus 186263\, \mathsf{L}^2\mathsf{S}_{12}$ by \cite[p. 491]{Getzler}.
\end{proof}

We expect that Theorem \ref{thm:15} can be improved as follows.
\begin{conj} \label{conj:15}
The STE generated by $H^{11}(\Mb_{1,11})$ contains $H^{15}(\Mb_{g,n})$ for $g \geq 2$. 
\end{conj}
The proof of Theorem \ref{thm:131415} in the next section shows that for $g\geq 2$, $H^{15}(\Mb_{g,n})^{\semis}$ contains no copies of $\mathsf{S}_{16}$. Conjecture \ref{conj:15} is therefore equivalent to the assertion that the STE generated by $H^{11}(\Mb_{1,11})$ and $H^{15}(\Mb_{1,15})$ contains $H^{15}(\Mb_{g,n})$ for all $g$ and $n$.

\section{Proof of Theorem \ref{thm:131415}}
The first two statements of Theorem \ref{thm:131415} follow immediately from Corollaries \ref{homologymotives14} and \ref{homologymotives13} by Poincar\'e duality. From Corollary \ref{homologymotives15}, we can similarly conclude that $H^{15}(\Mb_{g,n})^{\semis}$ is a direct sum of copies of $\mathsf{L}^2 \mathsf{S}_{12}$ and $\mathsf{S}_{16}$. To finish, it suffices to show that when $g\geq 2$, no copies of $\mathsf{S}_{16}$ appear. 

\begin{lem} \label{8.1}
The STE generated by $H^{11}(\Mb_{1,11})$ contains $H_{15}(\Mb_{2,n})$ for all $n$. Hence, $H^{15}(\Mb_{2,n})^{\mathrm{ss}} \cong \bigoplus \mathsf{L}^2\mathsf{S}_{12}$, as Galois representations or Hodge structures.
\end{lem}
\begin{proof}
When $g=2$, as Galois representations, there are no copies of $\mathsf{S}_{16}$ in $H^{15}(\Mb_{2,n})^{\semis}$ by Corollary \ref{hodge13g2}.
This shows that $H_{15}(\Mb_{2,n})^{\mathrm{ss}}$ is a direct sum of copies of $\mathsf{L}^{d_{2,n} - 13} \mathsf{S}_{12}$ in the category of Galois representations. But we know that $H_{15}(\Mb_{g,n})$ lies in the STE generated by $H^{11}(\Mb_{1,11})$ and $H^{15}(\Mb_{1,15})$ by Theorem \ref{thm:15}. By considering the Galois representations involved, it follows that $H_{15}(\Mb_{2,n})$ lies in the STE generated by $H^{11}(\Mb_{1,11})$. In particular, $H_{15}(\Mb_{2,n})^{\mathrm{ss}}$ is a direct sum of copies of $\mathsf{L}^{d_{2,n} - 13} \mathsf{S}_{12}$, either as Galois representations or Hodge structures. The second statement follows by Poincar\'e duality.
\end{proof}

To show there are no copies of $\mathsf{S}_{16}$ when $g\geq 3$, we will use a similar strategy to that of \cite{CanningLarsonPayne}. We
study the first two maps in the weight $15$ complex:
\begin{equation}\label{beta}
H^{15}(\Mb_{g,n})\xrightarrow{\alpha} \bigoplus_{|E(\Gamma)| = 1} H^{15}(\Mb_{\Gamma})^{\Aut(\Gamma)} \rightarrow
\bigoplus_{|E(\Gamma)| = 2} (H^{15}\big(\Mb_{\Gamma}) \otimes \det E(\Gamma)\big)^{\Aut(\Gamma)}.
\end{equation}
Here, $\Gamma$ is a stable graph of genus $g$ with $n$ legs. The first map is the pullback to the normalization of the boundary $\tilde{\partial \M_{g,n}}$. The second map is explicitly described in \cite[Section~4.1]{CanningLarsonPayne}. 

\begin{lem}\label{firstmap15}
If $g\geq 3$, then the pullback map $\alpha$ in \eqref{beta} is injective.
\end{lem}
\begin{proof}
The map $\alpha$ is injective when $\mathrm{gr}^W_{15}H^{15}_c(\M_{g,n})=0$. By \cite[Proposition 2.1]{BergstromFaberPayne}, this holds whenever $2g-2+n>15$. When $2g-2+n\leq 15$, $\M_{g,n}$ has the CKgP (see Table \ref{ckgptable}). The result follows from Proposition \ref{openCKgPresults}.
\end{proof}

We now assume $g\geq 3$, and we study the second map in \eqref{beta} as a morphism of Hodge structures. In particular, we consider the $H^{15,0}$ part:

\begin{equation}\label{hodgebeta}
\bigoplus_{|E(\Gamma)| = 1} H^{15,0}(\Mb_{\Gamma})^{\Aut(\Gamma)} \xrightarrow{ \ \beta \ } \bigoplus_{|E(\Gamma)| = 2} \big(H^{15,0}(\Mb_{\Gamma}) \otimes \det E(\Gamma)\big)^{\Aut(\Gamma)}.
\end{equation}

By Lemma \ref{firstmap15}, to prove the theorem in the category of Hodge structures, it suffices to show that $\beta$ is injective. The domain is a direct sum over graphs with one edge. If $\Gamma$ has one vertex and a loop, then $H^{15,0}(\Mb_{\Gamma})=H^{15,0}(\Mb_{g-1,n+2})=0$ by induction on $g$. If $\Gamma$ has two vertices with an edge connecting them, then by the K\"unneth formula, we have
\[
H^{15,0}(\Mb_{\Gamma})=H^{15,0}(\Mb_{a,A\cup p})\oplus H^{15,0}(\Mb_{b,A^c\cup q}).
\]
Assume $a\geq b$. By induction, the above is nonvanishing only if $b=1$ and $|A^c|\geq 14$. Because $g\geq 3$, it follows that $H^{15,0}(\Mb_{\Gamma})=H^{15,0}(\Mb_{1,A^c\cup q})$ in this case. Let $\Gamma'$ be the graph obtained from $\Gamma$ by adding a loop on the genus $g-1$ vertex and decreasing the genus accordingly, as in Figure \ref{Gamma2p}. Then $H^{15,0}(\Mb_{\Gamma'})$ has a summand $H^{15,0}(\Mb_{1,A^c\cup q})$, and the map $H^{15,0}(\Mb_{\Gamma})\rightarrow H^{15,0}(\Mb_{\Gamma'})$ injects into that summand. This completes the proof of Theorem \ref{thm:131415} in the category of Hodge structures. 
\begin{figure}[h!]
    \centering
    \includegraphics[width=3.2in]{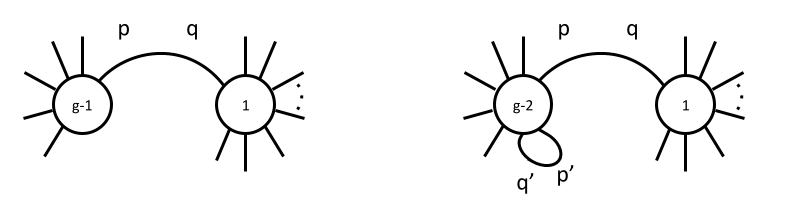}
    \caption{The graph $\Gamma$ on the left and $\Gamma'$ on the right. 
    }
    \label{Gamma2p}
\end{figure}

Finally, we explain how to deduce the result in the category of Galois representations from the result in the category of Hodge structures. To do so, we need one more lemma.

\begin{lem} \label{ds}
Let $V \subset H^{15}(\Mb_{1,15})$ be the subspace of classes pushed forward from the boundary and let $W \subset H^{15}(\Mb_{1,15})$ be the image of $H^{15}_c(\M_{1,15}) \to H^{15}(\Mb_{1,15})$. Then there is a direct sum decomposition $H^{15}(\Mb_{1,15}) = V \oplus W$,
as Hodge structures or as Galois representations.
Moreover, we have $W \cong \mathsf{S}_{16}$ and $V^{\semis} \cong 186263\, \mathsf{L}^2\mathsf{S}_{12}$.
\end{lem}
\begin{proof}
There is an exact sequence
\[0 \rightarrow V \rightarrow H^{15}(\Mb_{1,15}) \rightarrow W_{15}H^{15}(\M_{1,15}) \rightarrow 0.\]
It thus suffices to show that the composition
\begin{equation} \label{comp} W  \to H^{15}(\Mb_{1,15}) \to W_{15} H^{15}(\M_{1,15}) \cong \mathsf{S}_{16}
\end{equation}
is an isomorphism.
The pairing $H^{15}(\Mb_{1,15}) \times H^{15}(\Mb_{1,15}) \to \qq$ equips
$H^{15}(\Mb_{1,15})$ with a nondegenerate bilinear form under which the inclusion of $W$ is dual to the quotient $H^{15}(\Mb_{1,15}) \to W_{15} H^{15}(\M_{1,15})$.
To show \eqref{comp} is an isomorphism is therefore equivalent to showing that the restriction of the pairing to $W$ with itself is full rank. This becomes clear in the category of Hodge structures: Since $V$ has type $H^{13,2} \oplus H^{2,13}$ and $W$ has type $H^{15,0} \oplus H^{0,15}$, the pairing between $V$ and $W$ is trivial, so $W$ must be full rank on itself. This claim holds at the level of $\qq$-vector spaces. Since all the maps above are also maps of Galois representations, we have a direct sum decomposition as Galois representations too. 

The identification of the remainder
$V^{\semis} = 186263\, \mathsf{L}^2\mathsf{S}_{12}$
follows from \cite[p. 491]{Getzler}.
\end{proof}

By Poincar\'e duality, to prove Theorem \ref{thm:131415}, it suffices to show that when $g \geq 2$, there are no copies of $\mathsf{L}^{d_{g,n}-15} \mathsf{S}_{16}$ in the right-hand side of Corollary \ref{homologymotives15}. The proof of Theorem \ref{thm:15} exhibits $H^{2d_{g,n} - 15}(\Mb_{g,n})$ as a quotient of a direct sum of the form
\[\bigoplus \mathsf{L}^{d_{g,n}-13} \otimes H^{11}(\Mb_{1,11})\oplus \bigoplus \mathsf{L}^{d_{g,n}-15} \otimes H^{15}(\Mb_{1,15}) \twoheadrightarrow H^{2d_{g,n} - 15}(\Mb_{g,n}),  \]
where the map from each summand on the left is a composition of a forgetful pullback with a gluing pushforward. By Lemma \ref{ds}, it suffices to show that, for each of these maps, the composition with the inclusion of $W \subset H^{15}(\Mb_{1,15})$,
\begin{equation} \label{gm} \mathsf{L}^{d_{g,n}-15} \otimes W \rightarrow 
\mathsf{L}^{d_{g,n}-15} \otimes H^{15}(\Mb_{1,15}) \to
H^{2d_{g,n} - 15}(\Mb_{g,n}),
\end{equation}
vanishes.
When we view \eqref{gm} as a map of Hodge structures, we see that it vanishes since we have already proved Theorem \ref{thm:131415} in the category of Hodge structures. Thus, \eqref{gm} vanishes as a map of $\qq$-vector spaces, and so also as a map of Galois representations.
\qed

\bibliographystyle{amsplain}
\bibliography{refs}
\end{document}